\documentclass{article}
\usepackage{amsfonts}

\newtheorem{theorem}{Theorem}

\newtheorem{corollary}[theorem]{Corollary}

\newtheorem{lemma}[theorem]{Lemma}

\newtheorem{remark}[theorem]{Remark}

\newenvironment{proof}[1][Proof]{\noindent\textbf{#1.} }{\ \rule{0.5em}{0.5em}}
\input{tcilatex}
\begin{document}

\begin{center}
Some classes of generating functions for generalized Hermite- and
Chebyshev-type polynomials: Analysis of Euler's formula

\bigskip

Neslihan Kilar and Yilmaz Simsek

\bigskip

Department of Mathematics, Faculty of Science University of Akdeniz TR-07058
Antalya-TURKEY\\[0pt]

{E-mail: neslihankilar@gmail.com and ysimsek@akdeniz.edu.tr}

\bigskip
\end{center}

\textbf{Abstract }

\bigskip

The aim of this paper is to construct generating functions for new families
of special polynomials including the Appel polynomials, the Hermite-Kamp\`{e}
de F\`{e}riet polynomials, the Milne-Thomson type polynomials, parametric
kinds of Apostol type numbers and polynomials. Using Euler's formula,
relations among special functions, Hermite-type polynomials, the Chebyshev
polynomials and the Dickson polynomials are given. Using generating
functions and their functional equations, various formulas and identities
are given. With help of computational formula for new families of special
polynomials, some of their numerical values are given. Using hypegeometric
series, trigonometric functions and the Euler's formula, some applications
related to Hermite-type polynomials are presented. Finally, further remarks,
observations and comments about generating functions for new families of
special polynomials are given.



\section{Introduction}

The Euler's formula yields a important connection among analysis,
trigonometry and special functions. This formula also gives relations
between trigonometric functions and exponential functions. Because sine and
cosine functions are written as sums of the exponential functions.
Motivation of this paper is to construct generating functions for new
families of polynomials with the help of the Euler's formula. By using these
generating functions and their functional equations, new formulas,
identities, recurrence relations and properties of these polynomials, which
are the Appel polynomials, Apostol-type polynomials, Hermite-type
polynomials, the Chebyshev polynomials, the Dickson polynomials,
Milne-Thomson type polynomials, are given. Trigonometric functions, the
Euler's formula and generating functions have applications in many different
areas, which are mainly mathematics, statistics, physics, engineering and
other sciences. Therefore, it can be stated that the results of this article
may be used and applied in these related areas.

Notations and definitions of this paper are presented as follows:

Let $%
\mathbb{N}
$, $%
\mathbb{Z}
$, $%
\mathbb{R}
$, and $%
\mathbb{C}
$ denote the set of positive integers, the set of integers, the set of real
numbers, and the set of complex numbers, respectively, $%
\mathbb{N}
_{0}=%
\mathbb{N}
\cup \left\{ 0\right\} $. Let $\lambda $, $v\in 
\mathbb{C}
$. $\left( \alpha \right) _{n}$, denotes the Pochhammer symbol, is defined by%
\[
\left( \alpha \right) _{v}=\frac{\Gamma \left( \alpha +v\right) }{\Gamma
\left( \alpha \right) }=\left\{ 
\begin{array}{cc}
\alpha \left( \alpha +1\right) \left( \alpha +2\right) ...\left( \alpha
+v-1\right)  & v=n\in 
\mathbb{N}
\text{, }\alpha \in 
\mathbb{C}
\\ 
1 & v=0\text{, }\alpha \in 
\mathbb{C}
\setminus \left\{ 0\right\} 
\end{array}%
\right. 
\]%
where $\Gamma \left( \alpha \right) $\ dentes the Euler gamma function.%
\[
\binom{\alpha }{v}=\left\{ 
\begin{array}{cc}
\frac{\alpha \left( \alpha -1\right) \left( \alpha -2\right) ...\left(
\alpha -v+1\right) }{v!} & v\in 
\mathbb{N}
\text{, }\alpha \in 
\mathbb{C}
\\ 
1 & v=0%
\end{array}%
\right. 
\]%
and%
\[
\left( \alpha \right) ^{\underline{v}}=\left( -1\right) ^{v}\left( -\alpha
\right) _{v}.
\]%
Let%
\[
w=x+iy=(x,y)
\]%
and%
\[
\overline{w}=x-iy=(x,-y),
\]%
where%
\[
x=\func{Re}\left\{ w\right\} ,\text{ }y=\func{Im}\left\{ w\right\} 
\]%
and%
\[
i^{2}=-1.
\]%
Here, $\ln w$ takes its principal value such that%
\[
\ln (w):=\ln \left\vert w\right\vert +i\arg (w),
\]%
with $\left\vert w\right\vert >0$, $-\pi <\arg (w)<\pi $. In addition,%
\[
\exp \left( t\right) =e^{t}.
\]%
The Euler's formula, well-known mathematical formula in complex analysis, is
given by%
\[
\exp (iz)=\cos (z)+i\sin (z).
\]%
This formula gives the fundamental relationship between the trigonometric
functions and the complex exponential function (\textit{cf}. \cite{comtet}, 
\cite{Srivastava}, \cite{SrivastavaChoi2}).

The following generating functions for well-known numbers and polynomials
are needed in order give main results of this paper.

Generating function for the Apostol-Bernoulli polynomials $%
B_{n}^{(k)}(x;\lambda )$ of order $k$ is given by%
\begin{equation}
F_{AB}\left( t,x;\lambda ,k\right) =\left( \frac{t}{\lambda \exp \left(
t\right) -1}\right) ^{k}\exp \left( xt\right) =\sum_{n=0}^{\infty }\mathcal{B%
}_{n}^{(k)}(x;\lambda )\frac{t^{n}}{n!},  \label{ApostolB}
\end{equation}%
where $\left\vert t\right\vert <2\pi $, when $\lambda =1;$ $\left\vert
t\right\vert <\left\vert \log \lambda \right\vert $ when $\lambda \neq 1$.
Using (\ref{ApostolB}), we have%
\[
B_{n}^{\left( k\right) }\left( x\right) =\mathcal{B}_{n}^{(k)}(x;1) 
\]%
and%
\[
\mathcal{B}_{n}^{(k)}(\lambda )=\mathcal{B}_{n}^{(k)}(0;\lambda ), 
\]%
where $B_{n}^{\left( k\right) }\left( x\right) $ and $\mathcal{B}%
_{n}^{(k)}(\lambda )$ denote the Bernoulli polynomials of order $k$ and the
Apostol-Bernoulli numbers of order $k$, respectively (\textit{cf}. \cite%
{SrivastavaChoi2}, \cite{Srivastava2018}).

Generating function for the Apostol-Euler polynomials $\mathcal{E}%
_{n}^{(k)}(x;\lambda )$ of order $k$ is given by%
\begin{equation}
F_{AE}\left( t,x;\lambda ,k\right) =\left( \frac{2}{\lambda \exp \left(
t\right) +1}\right) ^{k}\exp \left( xt\right) =\sum_{n=0}^{\infty }\mathcal{E%
}_{n}^{(k)}(x;\lambda )\frac{t^{n}}{n!},  \label{ApostolE}
\end{equation}%
where $\left\vert t\right\vert <\left\vert \log \left( -\lambda \right)
\right\vert $. Using (\ref{ApostolE}), we have%
\[
E_{n}^{\left( k\right) }\left( x\right) =\mathcal{E}_{n}^{(k)}(x;1) 
\]%
and%
\[
\mathcal{E}_{n}^{(k)}(\lambda )=\mathcal{E}_{n}^{(k)}(0;\lambda ), 
\]%
where $E_{n}^{\left( k\right) }\left( x\right) $ and $\mathcal{E}%
_{n}^{(k)}(\lambda )$ denote the Euler polynomials of order $k$ and the
Apostol-Euler numbers of order $k$, respectively (\textit{cf}. \cite%
{SrivastavaChoi2}, \cite{Srivastava2018}).

Generating functions for the polynomials $C_{n}(x,y)$ and $S_{n}(x,y)$ are
defined as follows, respectively%
\begin{equation}
F_{C}\left( t,x,y\right) =\exp \left( xt\right) \cos \left( yt\right)
=\sum_{n=0}^{\infty }C_{n}(x,y)\frac{t^{n}}{n!},  \label{6a}
\end{equation}%
and%
\begin{equation}
F_{S}\left( t,x,y\right) =\exp \left( xt\right) \sin \left( yt\right)
=\sum_{n=0}^{\infty }S_{n}(x,y)\frac{t^{n}}{n!}  \label{6b}
\end{equation}%
(\textit{cf.} \cite{kilar}, \cite{KimRyoo}, \cite{masjed2016}, \cite%
{masjed2018}, \cite{masjedgenocchi}, \cite{Srivastava2018}).

By using equations (\ref{6a}) and (\ref{6b}), we have%
\begin{equation}
C_{n}(x,y)=\sum\limits_{k=0}^{\left[ \frac{n}{2}\right] }\left( -1\right)
^{k}\binom{n}{2k}x^{n-2k}y^{2k}  \label{7a}
\end{equation}%
and%
\begin{equation}
S_{n}(x,y)=\sum\limits_{k=0}^{\left[ \frac{n-1}{2}\right] }\left( -1\right)
^{k}\binom{n}{2k+1}x^{n-2k-1}y^{2k+1}  \label{7b}
\end{equation}%
(\textit{cf.} \cite{kilar}, \cite{KimRyoo}, \cite{masjed2016}, \cite%
{masjed2018}, \cite{masjedgenocchi}, \cite{Srivastava2018}).

Generating functions for the Chebyshev polynomials of the first and second
kinds\ are given as follows, respectively%
\begin{equation}
\frac{1-xt}{1-2xt+t^{2}}=\sum\limits_{n=0}^{\infty }T_{n}\left( x\right)
t^{n}  \label{Cp1}
\end{equation}%
and%
\begin{equation}
\frac{1}{1-2xt+t^{2}}=\sum\limits_{n=0}^{\infty }U_{n}\left( x\right) t^{n}
\label{Cp2}
\end{equation}%
(\textit{cf.} \cite{Abramowtz}, \cite{benjamin}, \cite{comtet}, \cite{Fox}, 
\cite{Rivlin}).

By using (\ref{Cp1}) and (\ref{Cp2}), the following well-known relations
between the polynomials $T_{n}\left( x\right) $ and $U_{n}\left( x\right) $
are given%
\begin{equation}
T_{n}\left( x\right) =U_{n}\left( x\right) -xU_{n-1}\left( x\right)
\label{TileU}
\end{equation}%
and%
\begin{equation}
T_{n+1}\left( x\right) =xT_{n}\left( x\right) -\left( 1-x^{2}\right)
U_{n-1}\left( x\right) .  \label{TileU2}
\end{equation}

By using (\ref{Cp1}) and (\ref{Cp2}), the well-known computational formulas
for the Chebyshev polynomials of the first and second kinds are given as
follows, respectively%
\begin{equation}
T_{n}\left( x\right) =\sum\limits_{k=0}^{\left[ \frac{n}{2}\right] }\binom{n%
}{2k}\left( x^{2}-1\right) ^{k}x^{n-2k}  \label{Ce1}
\end{equation}%
and%
\begin{equation}
U_{n-1}\left( x\right) =\sum\limits_{k=0}^{\left[ \frac{n-1}{2}\right] }%
\binom{n}{2k+1}\left( x^{2}-1\right) ^{k}x^{n-2k-1}  \label{Ce2}
\end{equation}%
(\textit{cf.} \cite{Abramowtz}, \cite{benjamin}, \cite{comtet}, \cite{Fox}, 
\cite{Rivlin}).

Generating functions for the Dickson polynomials of the first and second
kinds are given as follows, respectively%
\begin{equation}
\frac{1-2xt}{1-xt+\alpha t^{2}}=\sum\limits_{n=0}^{\infty }D_{n}\left(
x,\alpha \right) t^{n}  \label{d1}
\end{equation}%
and%
\begin{equation}
\frac{1}{1-xt+\alpha t^{2}}=\sum\limits_{n=0}^{\infty }\mathfrak{E}%
_{n}\left( x,\alpha \right) t^{n}  \label{d2}
\end{equation}%
(\textit{cf. }\cite{Dickson}, \cite{Lidl}, \cite{rasasisas}). The
polynomials $D_{n}\left( x,\alpha \right) $ and $\mathfrak{E}_{n}\left(
x,\alpha \right) $ are of degree $n$ in $x$ with real parameter $\alpha $.

By using (\ref{d1}) and (\ref{d2}), the following well-known relation
between the polynomials $D_{n}\left( x,\alpha \right) $ and $\mathfrak{E}%
_{n}\left( x,\alpha \right) $ is given%
\[
D_{n}\left( x,\alpha \right) =\mathfrak{E}_{n}\left( x,\alpha \right) -2x%
\mathfrak{E}_{n-1}\left( x,\alpha \right) . 
\]

Substituting $\alpha =1$ into (\ref{d1}) and (\ref{d2}), we have the
following relations, respectively:%
\begin{equation}
D_{n}\left( x,1\right) =2T_{n}\left( \frac{x}{2}\right) ,  \label{Dickson}
\end{equation}%
and%
\begin{equation}
\mathfrak{E}_{n}\left( x,1\right) =U_{n}\left( \frac{x}{2}\right)
\label{dickson2}
\end{equation}%
(\textit{cf. }\cite{Dickson}).

Generating functions for the Milne-Thomson type polynomials is given by%
\begin{equation}
R\left( t,x,y,z;a,b,v\right) =\left( b+f\left( t,a\right) \right) ^{z}\exp
\left( tx+yh\left( t,v\right) \right) =\sum\limits_{n=0}^{\infty
}y_{6}\left( n;x,y,z;a,b,v\right) \frac{t^{n}}{n!},  \label{y6}
\end{equation}%
where $f\left( t,a\right) $ is a number of family of analytic functions or
meromorphic functions, $h\left( t,v\right) $ any analytic function, $a,b\in 
\mathbb{R}
$ and $v\in 
\mathbb{N}
$ (\textit{cf}. \cite{SimsekRevista}).

Note that there is one generating function for each value of $a$, $b$ and $v$%
.

Substituting $y=0$ into (\ref{y6}), we have the Appell polynomials which are
defined by%
\[
\left( b+f\left( t,a\right) \right) ^{z}\exp \left( tx\right)
=\sum\limits_{n=0}^{\infty }y_{6}\left( n;x,0,z;a,b,v\right) \frac{t^{n}}{n!}%
,
\]%
where $\left( b+f\left( t,a\right) \right) ^{z}=\sum\limits_{n=0}^{\infty
}\vartheta _{n}^{(z)}(a,b)t^{n}$ is a formal power series and%
\[
y_{6}\left( n;x,0,z;a,b,v\right) =y_{6}\left( n;x,z;a,b\right) .
\]

Setting $x=y=0$ into (\ref{y6}), we obtain generating functions for special
numbers of order $z$:%
\begin{equation}
R_{1}\left( t,z;a,b\right) =\left( b+f\left( t,a\right) \right)
^{z}=\sum\limits_{n=0}^{\infty }y_{6}^{(z)}\left( n;a,b\right) \frac{t^{n}}{%
n!}.  \label{y66a}
\end{equation}%
Therefore, we have%
\[
y_{6}\left( n;0,0,z;a,b,v\right) =y_{6}^{(z)}\left( n;a,b\right) .
\]%
For instance, substituting $b=0$ and $f\left( t,a\right) =\frac{t}{a\exp
(t)-1}$ into (\ref{y6}), we have 
\[
y_{6}^{(z)}\left( n;a,0,v\right) =\mathcal{B}_{n}^{(z)}(a).
\]

Generating function for the Hermite-Kamp\`{e} de F\`{e}riet (or
Gould-Hopper) polynomials, $H_{n}^{\left( j\right) }\left( x,y\right) $ is
given by%
\begin{equation}
F_{H}\left( t,x,y,j\right) =\exp \left( xt+yt^{j}\right)
=\sum\limits_{n=0}^{\infty }H_{n}^{\left( j\right) }\left( x,y\right) \frac{%
t^{n}}{n!},  \label{H1}
\end{equation}%
where for $j\in 
\mathbb{N}
$ with $j\geq 2$,%
\[
H_{n}^{\left( j\right) }\left( x,y\right) =\sum\limits_{s=0}^{\left[ \frac{n%
}{2}\right] }\frac{x^{n-js}y^{s}}{\left( n-js\right) !s!}
\]%
(\textit{cf}. \cite{BrettiRicci}, \cite{Dattoli}, \cite{gould}, \cite%
{SrivastavaMonaccha}). It is well-known that the polynomials $H_{n}^{\left(
j\right) }\left( x,y\right) $ are a solution of generalized heat equation.

Generating function for generalized Hermite-Kamp\`{e} de F\`{e}riet
polynomials is given by%
\begin{equation}
F_{R}\left( t,\overrightarrow{u},r\right) =\exp \left(
\sum\limits_{j=1}^{r}u_{j}t^{j}\right) =\sum\limits_{n=0}^{\infty
}H_{n}\left( \overrightarrow{u},r\right) \frac{t^{n}}{n!},  \label{H2}
\end{equation}%
where for $\overrightarrow{u}=\left( u_{1},u_{2},\ldots ,u_{r}\right) $ and%
\begin{equation}
H_{n}\left( \overrightarrow{u},r\right) =\sum\limits_{\Pi _{m}\left( n\mid
r\right) }\frac{n!}{m_{1}!m_{2}!\cdots m_{r}!}\prod%
\limits_{j=1}^{r}u_{j}^{m_{j}}  \label{H2a}
\end{equation}%
such that%
\[
m=\sum\limits_{j=1}^{r}m_{j}, 
\]%
and%
\[
n=\sum\limits_{j=1}^{r}jm_{j} 
\]%
the sum (\ref{H2a}) runs over all restricted partitions $\Pi _{m}\left(
n\mid r\right) $ (containing at most $r$ sizes) of the integer $n$, $m$
denoting the number of parts of the partition and $m_{j}$ the number of
parts of size $j$ (\textit{cf}. see for detail \cite{BrettiRicci}, \cite%
{Dattoli}, \cite{datttoli2}).

Using equation (\ref{H2}), an explicit formula for the polynomials $%
H_{n}\left( \overrightarrow{u},r\right) $ is given by%
\[
H_{n}\left( \overrightarrow{u},r\right) =n!\sum\limits_{j=0}^{\left[ \frac{n%
}{r}\right] }\frac{u_{r}^{j}H_{n-rj}\left( \overrightarrow{u},r-1\right) }{%
j!\left( n-rj\right) !}
\]%
where $\left[ x\right] $ denote the largest integer $\leq x$. (\textit{cf}. 
\cite{Dattoli}, \cite{datttoli2}).

\section{Generating functions for new families of Hermite-type polynomials
and their computation formulas}

In this section, we define generating functions for families of Hermite-type
polynomials. We give some identities and computation formulas for these
polynomials and their generating functions.

Let%
\begin{equation}
\mathcal{G}\left( t,w,\overrightarrow{u},r\right) =\exp \left(
wt+\sum\limits_{j=1}^{r}u_{j}t^{j}\right) =\sum\limits_{n=0}^{\infty }%
\mathcal{K}\left( n;w,\overrightarrow{u},r\right) \frac{t^{n}}{n!},
\label{g1}
\end{equation}%
where $r$-tuples $\overrightarrow{u}=\left( u_{1},u_{2},\ldots ,u_{r}\right) 
$, $w=x+iy=(x,y)$, $u_{1},u_{2},\ldots ,u_{r}$, $x,y\in 
\mathbb{R}
$.

By combining equation (\ref{g1}) with the Euler's formula, we obtain%
\begin{equation}
\mathcal{G}\left( t,w,\overrightarrow{u},r\right) =\exp \left(
xt+\sum\limits_{j=1}^{r}u_{j}t^{j}\right) \left( \cos \left( yt\right)
+i\sin \left( yt\right) \right) =\sum\limits_{n=0}^{\infty }\mathcal{K}%
\left( n;w,\overrightarrow{u},r\right) \frac{t^{n}}{n!}.  \label{ag1}
\end{equation}%
In order to give an explicit formula for the polynomials $\mathcal{K}\left(
n;w,\overrightarrow{u},r\right) $, we give following decompositions of
equation (\ref{ag1})%
\begin{eqnarray}
K_{1}\left( t,x,y,\overrightarrow{u},r\right) &=&\func{Re}\left( \mathcal{G}%
\left( t,w,\overrightarrow{u},r\right) \right) =\exp \left(
xt+\sum\limits_{j=1}^{r}u_{j}t^{j}\right) \cos \left( yt\right)  \label{g1a}
\\
&=&\sum\limits_{n=0}^{\infty }k_{1}\left( n;x,y,\overrightarrow{u},r\right) 
\frac{t^{n}}{n!}  \nonumber
\end{eqnarray}%
and%
\begin{eqnarray}
K_{2}\left( t,x,y,\overrightarrow{u},r\right) &=&\func{Im}\left( \mathcal{G}%
\left( t,w,\overrightarrow{u},r\right) \right) =\exp \left(
xt+\sum\limits_{j=1}^{r}u_{j}t^{j}\right) \sin \left( yt\right)  \label{g2a}
\\
&=&\sum\limits_{n=0}^{\infty }k_{2}\left( n;x,y,\overrightarrow{u},r\right) 
\frac{t^{n}}{n!}.  \nonumber
\end{eqnarray}%
Therefore, by using (\ref{g1a}) and (\ref{g2a}), we get the following
decompositions for the polynomials $\mathcal{K}\left( n;w,\overrightarrow{u}%
,r\right) $:%
\begin{equation}
\mathcal{K}\left( n;w,\overrightarrow{u},r\right) =k_{1}\left( n;x,y,%
\overrightarrow{u},r\right) +ik_{2}\left( n;x,y,\overrightarrow{u},r\right) .
\label{g1f}
\end{equation}

\begin{lemma}
\label{LEMMA1}Let $\overrightarrow{x}=\left( x+u_{1},u_{2},u_{3},\ldots
,u_{r}\right) $ and $\overrightarrow{u}=\left( u_{1},u_{2},u_{3},\ldots
,u_{r}\right) $. Then we have%
\begin{equation}
k_{1}\left( n;x,y,\overrightarrow{u},r\right) =\sum\limits_{j=0}^{\left[ 
\frac{n}{2}\right] }\left( -1\right) ^{j}\binom{n}{2j}y^{2j}H_{n-2j}\left( 
\overrightarrow{x},r\right) .  \label{g1c}
\end{equation}
\end{lemma}

\begin{proof}
Combining (\ref{g1a}) with (\ref{H2}), we obtain the following functional
equation:%
\[
K_{1}\left( t,x,y,\overrightarrow{u},r\right) =\cos \left( yt\right)
F_{R}\left( t,\overrightarrow{x},r\right) 
\]%
where $\overrightarrow{x}=\left( x+u_{1},u_{2},u_{3},\ldots ,u_{r}\right) $
and $\overrightarrow{u}=\left( u_{1},u_{2},u_{3},\ldots ,u_{r}\right) $. By
using above functional equation, we get%
\[
\sum\limits_{n=0}^{\infty }k_{1}\left( n;x,y,\overrightarrow{u},r\right) 
\frac{t^{n}}{n!}=\sum\limits_{n=0}^{\infty }\left( -1\right) ^{n}y^{2n}\frac{%
t^{2n}}{\left( 2n\right) !}\sum\limits_{n=0}^{\infty }H_{n}\left( 
\overrightarrow{x},r\right) \frac{t^{n}}{n!}. 
\]%
Therefore%
\[
\sum\limits_{n=0}^{\infty }k_{1}\left( n;x,y,\overrightarrow{u},r\right) 
\frac{t^{n}}{n!}=\sum\limits_{n=0}^{\infty }\sum\limits_{j=0}^{\left[ \frac{n%
}{2}\right] }\left( -1\right) ^{j}\binom{n}{2j}y^{2j}H_{n-2j}\left( 
\overrightarrow{x},r\right) \frac{t^{n}}{n!}. 
\]%
Comparing the coefficients of $\frac{t^{n}}{n!}$ on both sides of the above
equation, we arrive at the desired result.
\end{proof}

By using (\ref{g1c}), we compute a few values of the polynomials $%
k_{1}\left( n;x,y,\overrightarrow{u},r\right) $ as follows:

For $r=2$, $\overrightarrow{u}=\left( u_{1},u_{2}\right) $ and $%
\overrightarrow{x}=\left( x+u_{1},u_{2}\right) $, we have%
\begin{eqnarray*}
k_{1}\left( 0;x,y,\overrightarrow{u},2\right) &=&1, \\
k_{1}\left( 1;x,y,\overrightarrow{u},2\right) &=&x+u_{1}, \\
k_{1}\left( 2;x,y,\overrightarrow{u},2\right) &=&\left( x+u_{1}\right)
^{2}+2u_{2}-y^{2}, \\
k_{1}\left( 3;x,y,\overrightarrow{u},2\right) &=&\left( x+u_{1}\right)
^{3}+6\left( x+u_{1}\right) u_{2}-3y^{2}\left( x+u_{1}\right) .
\end{eqnarray*}%
For $r=3$, $\overrightarrow{u}=\left( u_{1},u_{2},u_{3}\right) $ and $%
\overrightarrow{x}=\left( x+u_{1},u_{2},u_{3}\right) $, we have%
\begin{eqnarray*}
k_{1}\left( 0;x,y,\overrightarrow{u},3\right) &=&1, \\
k_{1}\left( 1;x,y,\overrightarrow{u},3\right) &=&x+u_{1}, \\
k_{1}\left( 2;x,y,\overrightarrow{u},3\right) &=&\left( x+u_{1}\right)
^{2}+2u_{2}-y^{2}, \\
k_{1}\left( 3;x,y,\overrightarrow{u},3\right) &=&\left( x+u_{1}\right)
^{3}+6\left( x+u_{1}\right) u_{2}+6u_{3}-3y^{2}\left( x+u_{1}\right) .
\end{eqnarray*}

\begin{lemma}
\label{LEMMA2}Let $\overrightarrow{x}=\left( x+u_{1},u_{2},u_{3},\ldots
,u_{r}\right) $ and $\overrightarrow{u}=\left( u_{1},u_{2},u_{3},\ldots
,u_{r}\right) $. Then we have%
\begin{equation}
k_{2}\left( n;x,y,\overrightarrow{u},r\right) =\sum\limits_{j=0}^{\left[ 
\frac{n-1}{2}\right] }\left( -1\right) ^{j}\binom{n}{2j+1}%
y^{2j+1}H_{n-1-2j}\left( \overrightarrow{x},r\right) .  \label{g1d}
\end{equation}
\end{lemma}

\begin{proof}
Combining (\ref{g2a}) with (\ref{H2}), we get the following functional
equation:%
\[
K_{2}\left( t,x,y,\overrightarrow{u},r\right) =\sin \left( yt\right)
F_{R}\left( t,\overrightarrow{x},r\right) 
\]%
where $\overrightarrow{x}=\left( x+u_{1},u_{2},u_{3},\ldots ,u_{r}\right) $
and $\overrightarrow{u}=\left( u_{1},u_{2},u_{3},\ldots ,u_{r}\right) $. By
using above functional equation, we have%
\[
\sum\limits_{n=0}^{\infty }k_{2}\left( n;x,y,\overrightarrow{u},r\right) 
\frac{t^{n}}{n!}=\sum\limits_{n=0}^{\infty }\left( -1\right) ^{n}y^{2n+1}%
\frac{t^{2n+1}}{\left( 2n+1\right) !}\sum\limits_{n=0}^{\infty }H_{n}\left( 
\overrightarrow{x},r\right) \frac{t^{n}}{n!}. 
\]%
Therefore%
\[
\sum\limits_{n=0}^{\infty }k_{2}\left( n;x,y,\overrightarrow{u},r\right) 
\frac{t^{n}}{n!}=\sum\limits_{n=0}^{\infty }\sum\limits_{j=0}^{\left[ \frac{%
n-1}{2}\right] }\left( -1\right) ^{j}\binom{n}{2j+1}y^{2j+1}H_{n-1-2j}\left( 
\overrightarrow{x},r\right) \frac{t^{n}}{n!}. 
\]%
Comparing the coefficients of $\frac{t^{n}}{n!}$ on both sides of the above
equation, we arrive at the desired result.
\end{proof}

By using (\ref{g1d}), we compute a few values of the polynomials $%
k_{2}\left( n;x,y,\overrightarrow{u},r\right) $ as follows:

For $r=2$, $\overrightarrow{u}=\left( u_{1},u_{2}\right) $ and $%
\overrightarrow{x}=\left( x+u_{1},u_{2}\right) $, we have%
\begin{eqnarray*}
k_{2}\left( 0;x,y,\overrightarrow{u},2\right) &=&0, \\
k_{2}\left( 1;x,y,\overrightarrow{u},2\right) &=&y, \\
k_{2}\left( 2;x,y,\overrightarrow{u},2\right) &=&2y\left( x+u_{1}\right) , \\
k_{2}\left( 3;x,y,\overrightarrow{u},2\right) &=&3y\left( x+u_{1}\right)
^{2}+6yu_{2}-y^{3}.
\end{eqnarray*}%
For $r=3$, $\overrightarrow{u}=\left( u_{1},u_{2},u_{3}\right) $ and $%
\overrightarrow{x}=\left( x+u_{1},u_{2},u_{3}\right) $, we have%
\begin{eqnarray*}
k_{2}\left( 0;x,y,\overrightarrow{u},3\right) &=&0, \\
k_{2}\left( 1;x,y,\overrightarrow{u},3\right) &=&y, \\
k_{2}\left( 2;x,y,\overrightarrow{u},3\right) &=&2y\left( x+u_{1}\right) , \\
k_{2}\left( 3;x,y,\overrightarrow{u},3\right) &=&3y\left( x+u_{1}\right)
^{2}+6yu_{2}-y^{3}.
\end{eqnarray*}%
Combining Lemma \ref{LEMMA1} and Lemma \ref{LEMMA2} with (\ref{g1f}), we
obtain an explicit formula for the polynomials $\mathcal{K}\left( n;w,%
\overrightarrow{u},r\right) $ by the following theorem:

\begin{theorem}
Let $\overrightarrow{x}=\left( x+u_{1},u_{2},u_{3},\ldots ,u_{r}\right) $
and $\overrightarrow{u}=\left( u_{1},u_{2},u_{3},\ldots ,u_{r}\right) $.
Then we have%
\begin{equation}
\mathcal{K}\left( n;w,\overrightarrow{u},r\right) =\sum\limits_{j=0}^{\left[ 
\frac{n}{2}\right] }\left( -1\right) ^{j}\binom{n}{2j}y^{2j}H_{n-2j}\left( 
\overrightarrow{x},r\right) +i\sum\limits_{j=0}^{\left[ \frac{n-1}{2}\right]
}\left( -1\right) ^{j}\binom{n}{2j+1}y^{2j+1}H_{n-1-2j}\left( 
\overrightarrow{x},r\right) .  \label{d1fg}
\end{equation}
\end{theorem}

By using (\ref{d1fg}), we compute a few values of the polynomials $\mathcal{K%
}\left( n;w,\overrightarrow{u},r\right) $ as follows:

For $r=2$, $\overrightarrow{u}=\left( u_{1},u_{2}\right) $ and $%
\overrightarrow{x}=\left( x+u_{1},u_{2}\right) $, we have%
\begin{eqnarray*}
\mathcal{K}\left( 0;w,\overrightarrow{u},2\right) &=&1, \\
\mathcal{K}\left( 1;w,\overrightarrow{u},2\right) &=&x+u_{1}+iy, \\
\mathcal{K}\left( 2;w,\overrightarrow{u},2\right) &=&\left( x+u_{1}\right)
^{2}+2u_{2}-y^{2}+2iy\left( x+u_{1}\right) , \\
\mathcal{K}\left( 3;w,\overrightarrow{u},2\right) &=&\left( x+u_{1}\right)
^{3}+6\left( x+u_{1}\right) u_{2}-3y^{2}\left( x+u_{1}\right) +i\left(
3y\left( x+u_{1}\right) ^{2}+6yu_{2}-y^{3}\right) .
\end{eqnarray*}%
For $r=3$, $\overrightarrow{u}=\left( u_{1},u_{2},u_{3}\right) $ and $%
\overrightarrow{x}=\left( x+u_{1},u_{2},u_{3}\right) $, we have%
\begin{eqnarray*}
\mathcal{K}\left( 0;w,\overrightarrow{u},3\right) &=&1, \\
\mathcal{K}\left( 1;w,\overrightarrow{u},3\right) &=&x+u_{1}+iy, \\
\mathcal{K}\left( 2;w,\overrightarrow{u},3\right) &=&\left( x+u_{1}\right)
^{2}+2u_{2}-y^{2}+2iy\left( x+u_{1}\right) , \\
\mathcal{K}\left( 3;w,\overrightarrow{u},3\right) &=&\left( x+u_{1}\right)
^{3}+6\left( x+u_{1}\right) u_{2}+6u_{3}-3y^{2}\left( x+u_{1}\right)
+i\left( 3y\left( x+u_{1}\right) ^{2}+6yu_{2}-y^{3}\right) .
\end{eqnarray*}

\begin{theorem}
Let $\overrightarrow{u}=\left( u_{1},u_{2},u_{3},\ldots ,u_{r}\right) $.
Then we have%
\begin{equation}
k_{1}\left( n;x,y,\overrightarrow{u},r\right) =\sum\limits_{j=0}^{n}\binom{n%
}{j}C_{j}\left( x,y\right) H_{n-j}\left( \overrightarrow{u},r\right) .
\label{K1HC}
\end{equation}
\end{theorem}

\begin{proof}
By using (\ref{6a}), (\ref{H2}) and (\ref{g1a}), we obtain the following
functional equation:%
\[
K_{1}\left( t,x,y,\overrightarrow{u},r\right) =F_{C}\left( t,x,y\right)
F_{R}\left( t,\overrightarrow{u},r\right) . 
\]%
By using the above functional equation, we get%
\[
\sum\limits_{n=0}^{\infty }k_{1}\left( n;x,y,\overrightarrow{u},r\right) 
\frac{t^{n}}{n!}=\sum\limits_{n=0}^{\infty }C_{n}\left( x,y\right) \frac{%
t^{n}}{n!}\sum\limits_{n=0}^{\infty }H_{n}\left( \overrightarrow{u},r\right) 
\frac{t^{n}}{n!}. 
\]%
Therefore%
\[
\sum\limits_{n=0}^{\infty }k_{1}\left( n;x,y,\overrightarrow{u},r\right) 
\frac{t^{n}}{n!}=\sum\limits_{n=0}^{\infty }\sum\limits_{j=0}^{n}\binom{n}{j}%
C_{j}\left( x,y\right) H_{n-j}\left( \overrightarrow{u},r\right) \frac{t^{n}%
}{n!}. 
\]%
Comparing the coefficients of $\frac{t^{n}}{n!}$ on both sides of the above
equation, we arrive at the desired result.
\end{proof}

\begin{theorem}
Let $\overrightarrow{u}=\left( u_{1},u_{2},u_{3},\ldots ,u_{r}\right) $.
Then we have%
\begin{equation}
k_{2}\left( n;x,y,\overrightarrow{u},r\right) =\sum\limits_{j=0}^{n}\binom{n%
}{j}S_{j}\left( x,y\right) H_{n-j}\left( \overrightarrow{u},r\right) .
\label{K2HS}
\end{equation}
\end{theorem}

\begin{proof}
By using (\ref{6b}), (\ref{H2}) and (\ref{g2a}), we derive the following
functional equation:%
\[
K_{2}\left( t,x,y,\overrightarrow{u},r\right) =F_{S}\left( t,x,y\right)
F_{R}\left( t,\overrightarrow{u},r\right) . 
\]%
By using above functional equation, we get%
\[
\sum\limits_{n=0}^{\infty }k_{2}\left( n;x,y,\overrightarrow{u},r\right) 
\frac{t^{n}}{n!}=\sum\limits_{n=0}^{\infty }S_{n}\left( x,y\right) \frac{%
t^{n}}{n!}\sum\limits_{n=0}^{\infty }H_{n}\left( \overrightarrow{u},r\right) 
\frac{t^{n}}{n!}. 
\]%
Therefore%
\[
\sum\limits_{n=0}^{\infty }k_{2}\left( n;x,y,\overrightarrow{u},r\right) 
\frac{t^{n}}{n!}=\sum\limits_{n=0}^{\infty }\sum\limits_{j=0}^{n}\binom{n}{j}%
S_{j}\left( x,y\right) H_{n-j}\left( \overrightarrow{u},r\right) \frac{t^{n}%
}{n!}. 
\]%
Comparing the coefficients of $\frac{t^{n}}{n!}$ on both sides of the above
equation, we arrive at the desired result.
\end{proof}

Combining (\ref{K1HC}) and (\ref{K2HS}) with (\ref{g1f}), we obtain an
explicit formula for the polynomials $\mathcal{K}\left( n;w,\overrightarrow{u%
},r\right) $ by the following corollary:

\begin{corollary}
Let $\overrightarrow{u}=\left( u_{1},u_{2},u_{3},\ldots ,u_{r}\right) $.
Then we have%
\[
\mathcal{K}\left( n;w,\overrightarrow{u},r\right) =\sum\limits_{j=0}^{n}%
\binom{n}{j}H_{n-j}\left( \overrightarrow{u},r\right) \left( C_{j}\left(
x,y\right) +iS_{j}\left( x,y\right) \right) . 
\]
\end{corollary}

\section{Generating functions for Hermite-based $r$-parametric
Milne-Thomson-type polynomials}

By the aid of generating functions in (\ref{y6}) and (\ref{g1}), we
construct the following generating functions for Hermite-based $r$%
-parametric Milne-Thomson-type polynomials:%
\begin{eqnarray}
M_{1}\left( t,w,z,\overrightarrow{u},r,a,b\right) &=&\left( b+f\left(
t,a\right) \right) ^{z}\mathcal{G}\left( t,w,\overrightarrow{u},r\right)
\label{M1} \\
&=&\sum\limits_{n=0}^{\infty }h\left( n,w,z;\overrightarrow{u},r,a,b\right) 
\frac{t^{n}}{n!},  \nonumber
\end{eqnarray}%
\begin{eqnarray}
M_{2}\left( t,w,z,\overrightarrow{u},r,a,b\right) &=&\left( b+f\left(
t,a\right) \right) ^{z}\left( \mathcal{G}\left( t,w,\overrightarrow{u}%
,r\right) +\mathcal{G}\left( t,\overline{w},\overrightarrow{u},r\right)
\right)  \label{M2} \\
&=&\sum\limits_{n=0}^{\infty }h_{1}\left( n,w,z;\overrightarrow{u}%
,r,a,b\right) \frac{t^{n}}{n!},  \nonumber
\end{eqnarray}%
and%
\begin{eqnarray}
M_{3}\left( t,w,z,\overrightarrow{u},r,a,b\right) &=&\left( b+f\left(
t,a\right) \right) ^{z}\left( \mathcal{G}\left( t,w,\overrightarrow{u}%
,r\right) -\mathcal{G}\left( t,\overline{w},\overrightarrow{u},r\right)
\right)  \label{M3} \\
&=&\sum\limits_{n=0}^{\infty }h_{2}\left( n,w,z;\overrightarrow{u}%
,r,a,b\right) \frac{t^{n}}{n!}.  \nonumber
\end{eqnarray}%
where $a,b,z\in \mathbb{R}$, $r$-tuples $\overrightarrow{u}=\left(
u_{1},u_{2},\ldots ,u_{r}\right) $, $w=x+iy$ and $\overline{w}=x-iy$; the
function $f\left( t,a\right) $ denotes analytic or meromorphic function.

Substituting $w=0$ and $\overrightarrow{u}=\overrightarrow{0}$ into (\ref{M1}%
), we have%
\[
\sum\limits_{n=0}^{\infty }h\left( n,0,z;\overrightarrow{0},r,a,b\right) 
\frac{t^{n}}{n!}=\sum\limits_{n=0}^{\infty }y_{6}^{(z)}\left( n;a,b\right) 
\frac{t^{n}}{n!}.
\]%
Thus, we get%
\[
h\left( n,0,z;\overrightarrow{0},r,a,b\right) =y_{6}^{(z)}\left(
n;a,b\right) .
\]

Combining (\ref{M3}) and (\ref{M2}) with (\ref{M1}), we obtain an formula
for the polynomials $h\left( n,w,z;\overrightarrow{u},r,a,b\right) $ by the
following theorem:

\begin{theorem}
Let $\overrightarrow{u}=\left( u_{1},u_{2},u_{3},\ldots ,u_{r}\right) $ and $%
w=x+iy$. Then we have%
\begin{equation}
h\left( n,w,z;\overrightarrow{u},r,a,b\right) =\frac{h_{1}\left( n,w,z;%
\overrightarrow{u},r,a,b\right) +h_{2}\left( n,w,z;\overrightarrow{u}%
,r,a,b\right) }{2}.  \label{1yH}
\end{equation}
\end{theorem}

\begin{theorem}
Let $\overrightarrow{u}=\left( u_{1},u_{2},u_{3},\ldots ,u_{r}\right) $.
Then we have%
\begin{equation}
h_{1}\left( n,w,z;\overrightarrow{u},r,a,b\right) =\sum\limits_{j=0}^{n}%
\binom{n}{j}y_{6}^{(z)}\left( n-j;a,b\right) \left( \mathcal{K}\left( j;w,%
\overrightarrow{u},r\right) +\mathcal{K}\left( j;\overline{w},%
\overrightarrow{u},r\right) \right) .  \label{NK1}
\end{equation}
\end{theorem}

\begin{proof}
By using (\ref{y66a}), (\ref{g1}) and (\ref{M2}), we derive the following
functional equation:%
\[
M_{2}\left( t,w,z,\overrightarrow{u},r,a,b\right) =R_{1}\left(
t,z;a,b\right) \left( \mathcal{G}\left( t,w,\overrightarrow{u},r\right) +%
\mathcal{G}\left( t,\overline{w},\overrightarrow{u},r\right) \right) .
\]%
From the above equation, we have%
\begin{eqnarray*}
\sum\limits_{n=0}^{\infty }h_{1}\left( n,w,z;\overrightarrow{u},r,a,b\right) 
\frac{t^{n}}{n!} &=&\sum\limits_{n=0}^{\infty }y_{6}^{(z)}\left(
n;a,b\right) \frac{t^{n}}{n!}\sum\limits_{n=0}^{\infty }\mathcal{K}\left(
n;w,\overrightarrow{u},r\right) \frac{t^{n}}{n!} \\
&&+\sum\limits_{n=0}^{\infty }y_{6}^{(z)}\left( n;a,b\right) \frac{t^{n}}{n!}%
\sum\limits_{n=0}^{\infty }\mathcal{K}\left( n;\overline{w},\overrightarrow{u%
},r\right) \frac{t^{n}}{n!}.
\end{eqnarray*}%
Therefore%
\begin{eqnarray*}
&&\sum\limits_{n=0}^{\infty }h_{1}\left( n,w,z;\overrightarrow{u}%
,r,a,b\right) \frac{t^{n}}{n!} \\
&=&\sum\limits_{n=0}^{\infty }\sum\limits_{j=0}^{n}\binom{n}{j}%
y_{6}^{(z)}\left( n-j;a,b\right) \left( \mathcal{K}\left( j;w,%
\overrightarrow{u},r\right) +\mathcal{K}\left( j;\overline{w},%
\overrightarrow{u},r\right) \right) \frac{t^{n}}{n!}.
\end{eqnarray*}%
Comparing the coefficients of $\frac{t^{n}}{n!}$ on both sides of the above
equation we arrive at the desired result.
\end{proof}

\begin{theorem}
Let $\overrightarrow{u}=\left( u_{1},u_{2},u_{3},\ldots ,u_{r}\right) $ and $%
w=x+iy$. Then we have%
\begin{equation}
h_{2}\left( n,w,z;\overrightarrow{u},r,a,b\right) =\sum\limits_{j=0}^{n}%
\binom{n}{j}y_{6}^{(z)}\left( n-j;a,b\right) \left( \mathcal{K}\left( j;w,%
\overrightarrow{u},r\right) -\mathcal{K}\left( j;\overline{w},%
\overrightarrow{u},r\right) \right) .  \label{NK2}
\end{equation}
\end{theorem}

\begin{proof}
By using (\ref{y6}), (\ref{g1}) and (\ref{M3}), we derive the following
functional equation:%
\[
M_{3}\left( t,w,z,\overrightarrow{u},r,a,b\right) =R_{1}\left(
t,z;a,b\right) \left( \mathcal{G}\left( t,w,\overrightarrow{u},r\right) -%
\mathcal{G}\left( t,\overline{w},\overrightarrow{u},r\right) \right) .
\]%
From the above equation, we obtain%
\begin{eqnarray*}
\sum\limits_{n=0}^{\infty }h_{2}\left( n,w,z;\overrightarrow{u},r,a,b\right) 
\frac{t^{n}}{n!} &=&\sum\limits_{n=0}^{\infty }y_{6}^{(z)}\left(
n;a,b\right) \frac{t^{n}}{n!}\sum\limits_{n=0}^{\infty }\mathcal{K}\left(
n;w,\overrightarrow{u},r\right) \frac{t^{n}}{n!} \\
&&-\sum\limits_{n=0}^{\infty }y_{6}^{(z)}\left( n;a,b\right) \frac{t^{n}}{n!}%
\sum\limits_{n=0}^{\infty }\mathcal{K}\left( n;\overline{w},\overrightarrow{u%
},r\right) \frac{t^{n}}{n!}.
\end{eqnarray*}%
Therefore%
\begin{eqnarray*}
&&\sum\limits_{n=0}^{\infty }h_{2}\left( n,w,z;\overrightarrow{u}%
,r,a,b\right) \frac{t^{n}}{n!} \\
&=&\sum\limits_{n=0}^{\infty }\sum\limits_{j=0}^{n}\binom{n}{j}%
y_{6}^{(z)}\left( n-j;a,b\right) \left( \mathcal{K}\left( j;w,%
\overrightarrow{u},r\right) -\mathcal{K}\left( j;\overline{w},%
\overrightarrow{u},r\right) \right) \frac{t^{n}}{n!}.
\end{eqnarray*}%
Comparing the coefficients of $\frac{t^{n}}{n!}$ on both sides of the above
equation we arrive at the desired result.
\end{proof}

\subsection{Identities for Hermite-based $r$-parametric Milne-Thomson-type
polynomials}

By using (\ref{M1})-(\ref{M3}), we give identities and relations for
Milne-Thomson type polynomials and Hermite-type polynomials including
Hermite-based $r$-parametric Milne-Thomson-type polynomials.

\begin{theorem}
Let $\overrightarrow{u}=\left( u_{1},u_{2},u_{3},\ldots ,u_{r}\right) $ and $%
w=x+iy$. Then we have%
\begin{equation}
h\left( n,w,z;\overrightarrow{u},r,a,b\right) =\sum\limits_{j=0}^{n}\binom{n%
}{j}y_{6}^{(z)}\left( n-j;a,b\right) \mathcal{K}\left( j;w,\overrightarrow{u}%
,r\right) .  \label{NK3}
\end{equation}
\end{theorem}

\begin{proof}
By using (\ref{y66a}), (\ref{g1}) and (\ref{M1}), we derive the following
functional equation:%
\[
M_{1}\left( t,w,z,\overrightarrow{u},r,a,b\right) =R_{1}\left(
t,z;a,b\right) \mathcal{G}\left( t,w,\overrightarrow{u},r\right) .
\]%
From the above equation, we have%
\[
\sum\limits_{n=0}^{\infty }h\left( n,w,z;\overrightarrow{u},r,a,b\right) 
\frac{t^{n}}{n!}=\sum\limits_{n=0}^{\infty }y_{6}^{(z)}\left( n;a,b\right) 
\frac{t^{n}}{n!}\sum\limits_{n=0}^{\infty }\mathcal{K}\left( n;w,%
\overrightarrow{u},r\right) \frac{t^{n}}{n!}.
\]%
Therefore%
\begin{eqnarray*}
&&\sum\limits_{n=0}^{\infty }h\left( n,w,z;\overrightarrow{u},r,a,b\right) 
\frac{t^{n}}{n!} \\
&=&\sum\limits_{n=0}^{\infty }\sum\limits_{j=0}^{n}\binom{n}{j}%
y_{6}^{(z)}\left( n-j;a,b\right) \mathcal{K}\left( j;w,\overrightarrow{u}%
,r\right) \frac{t^{n}}{n!}.
\end{eqnarray*}%
Comparing the coefficients of $\frac{t^{n}}{n!}$ on both sides of the above
equation we arrive at the desired result.
\end{proof}

By using Euler's formula, we modify (\ref{M2}) as follows:%
\begin{eqnarray}
B\left( t,x,y,z,\overrightarrow{u},r,a,b\right)  &=&2\left( b+f\left(
t,a\right) \right) ^{z}\exp \left( xt\right) M_{4}\left( t,y,\overrightarrow{%
u},r\right)   \label{M2a} \\
&=&\sum\limits_{n=0}^{\infty }\mathfrak{h}_{1}\left( n,x,y,z;\overrightarrow{%
u},r,a,b\right) \frac{t^{n}}{n!},  \nonumber
\end{eqnarray}%
where%
\begin{equation}
M_{4}\left( t,y,\overrightarrow{u},r\right) =\exp \left(
\sum\limits_{j=1}^{r}u_{j}t^{j}\right) \cos \left( yt\right)
=\sum\limits_{n=0}^{\infty }C_{n}\left( \overrightarrow{u},y;r\right) \frac{%
t^{n}}{n!}.  \label{1Y2c}
\end{equation}%
Observe that when $r=1$, (\ref{1Y2c}) reduces to the (\ref{6a}). Setting $y=0
$ in (\ref{1Y2c}), we have%
\[
H_{n}\left( \overrightarrow{u},r\right) =C_{n}\left( \overrightarrow{u}%
,0;r\right) .
\]

\begin{theorem}
Let $\overrightarrow{u}=\left( u_{1},u_{2},u_{3},\ldots ,u_{r}\right) $.
Then we have%
\begin{equation}
\mathfrak{h}_{1}\left( n,x,y,z;\overrightarrow{u},r,a,b\right)
=2\sum\limits_{j=0}^{n}\binom{n}{j}y_{6}\left( n-j;x,z;a,b\right)
C_{j}\left( \overrightarrow{u},y;r\right) .  \label{2.3}
\end{equation}
\end{theorem}

\begin{proof}
Combining (\ref{y6}), (\ref{1Y2c}) and (\ref{M2a}), we get%
\[
\sum\limits_{n=0}^{\infty }\mathfrak{h}_{1}\left( n,x,y,z;\overrightarrow{u}%
,r,a,b\right) \frac{t^{n}}{n!}=2\sum\limits_{n=0}^{\infty }y_{6}\left(
n;x,z;a,b\right) \frac{t^{n}}{n!}\sum\limits_{n=0}^{\infty }C_{n}\left( 
\overrightarrow{u},y;r\right) \frac{t^{n}}{n!}.
\]%
Therefore%
\[
\sum\limits_{n=0}^{\infty }\mathfrak{h}_{1}\left( n,x,y,z;\overrightarrow{u}%
,r,a,b\right) \frac{t^{n}}{n!}=2\sum\limits_{n=0}^{\infty
}\sum\limits_{j=0}^{n}\binom{n}{j}y_{6}\left( n-j;x,z;a,b\right) C_{j}\left( 
\overrightarrow{u},y;r\right) \frac{t^{n}}{n!}.
\]%
Comparing the coefficients of $\frac{t^{n}}{n!}$ on both sides of the above
equation we arrive at the desired result.
\end{proof}

Substituting $r=1$ into (\ref{2.3}), we have%
\[
\mathfrak{h}_{1}\left( n,x,y,z;u_{1},1,a,b\right) =2\sum\limits_{j=0}^{n}%
\binom{n}{j}y_{6}\left( n-j;x,z;a,b\right) C_{j}\left( u_{1},y;1\right) ,
\]%
where%
\[
C_{j}(u_{1},y)=C_{j}\left( u_{1},y;1\right) .
\]

Substituting $b=0$,%
\[
f\left( t,a\right) =\frac{t}{a\exp \left( t\right) -1}
\]%
and $\overrightarrow{u}=\overrightarrow{0}$ into (\ref{M2a}), we have%
\begin{equation}
\sum\limits_{n=0}^{\infty }\mathfrak{h}_{1}\left( n,x,y,z;\overrightarrow{0}%
,r,a,0\right) \frac{t^{n}}{n!}=\sum\limits_{n=0}^{\infty }2\mathcal{B}%
_{n}^{\left( C,z\right) }\left( x,y;a\right) \frac{t^{n}}{n!},  \label{Bernl}
\end{equation}%
where%
\[
\left( \frac{t}{a\exp \left( t\right) -1}\right) ^{z}\exp \left( xt\right)
\cos \left( yt\right) =\sum\limits_{n=0}^{\infty }\mathcal{B}_{n}^{\left(
C,z\right) }\left( x,y;a\right) \frac{t^{n}}{n!}
\]%
where the polynomials $\mathcal{B}_{n}^{\left( C,z\right) }\left(
x,y;a\right) $ were defined by Srivastava at al (\textit{cf}. \cite%
{kizilates}). Comparing the coefficients of $\frac{t^{n}}{n!}$ on both sides
of equation (\ref{Bernl}), we get the following result:

\begin{corollary}
\[
\mathfrak{h}_{1}\left( n,x,y,z;\overrightarrow{0},r,a,0\right) =2\mathcal{B}%
_{n}^{\left( C,z\right) }\left( x,y;a\right) .
\]
\end{corollary}

\begin{remark}
When $a=1$ and $z=1$, the polynomials $\mathfrak{h}_{1}\left( n,x,y,z;%
\overrightarrow{0},r,a,0\right) $ reduce to following well-known polynomials:%
\[
\mathfrak{h}_{1}\left( n,x,y,1;\overrightarrow{0},r,1,0\right) =2\mathcal{B}%
_{n}^{\left( C,1\right) }\left( x,y;1\right) =2B_{n}^{\left( C\right)
}\left( x,y\right) ,
\]%
where the polynomials $B_{n}^{\left( C\right) }\left( x,y\right) $ denote
the cosine-Bernoulli polynomials (\textit{cf}. \cite{KimRyoo}). When $y=0$,
the polynomials $\mathcal{B}_{n}^{\left( C,z\right) }\left( x,y;a\right) $
reduce to the Apostol-Bernoulli polynomials of order $z$:%
\[
\mathcal{B}_{n}^{\left( z\right) }\left( x;a\right) =\mathcal{B}_{n}^{\left(
C,z\right) }\left( x,0;a\right) 
\]%
(\textit{cf}. \cite{SrivastavaChoi2}, \cite{Srivastava2018}).
\end{remark}

On the other hand, using (\ref{MBC2}), we have%
\[
\sum\limits_{n=0}^{\infty }\mathfrak{h}_{1}\left( n,x,y,z;\overrightarrow{0}%
,r,a,0\right) \frac{t^{n}}{n!}=2\sum\limits_{n=0}^{\infty }\mathcal{B}%
_{n}^{\left( z\right) }\left( x;a\right) \frac{t^{n}}{n!}\sum\limits_{n=0}^{%
\infty }(-1)^{n}\frac{\left( yt\right) ^{2n}}{\left( 2n\right) !}.
\]

Therefore, we obtain%
\[
\sum\limits_{n=0}^{\infty }\mathfrak{h}_{1}\left( n,x,y,z;\overrightarrow{0}%
,r,a,0\right) \frac{t^{n}}{n!}=\sum\limits_{n=0}^{\infty }\left(
2\sum\limits_{n=0}^{\left[ \frac{n}{2}\right] }(-1)^{j}\binom{n}{2j}y^{2j}%
\mathcal{B}_{n-2j}^{\left( z\right) }\left( x;a\right) \right) \frac{t^{n}}{%
n!}.
\]%
Comparing the coefficients of $\frac{t^{n}}{n!}$ on both sides of the above
equation, we arrive at the following result:

\begin{corollary}
Let $r$-tuples $\overrightarrow{0}=(0,0,...,0)$. Then we have%
\[
\mathfrak{h}_{1}\left( n,x,y,z;\overrightarrow{0},r,a,0\right)
=2\sum\limits_{n=0}^{\left[ \frac{n}{2}\right] }(-1)^{j}\binom{n}{2j}y^{2j}%
\mathcal{B}_{n-2j}^{\left( z\right) }\left( x;a\right) .
\]
\end{corollary}

We modify (\ref{MBC2}) as follows:%
\begin{equation}
\sum\limits_{n=0}^{\infty }\mathfrak{h}_{1}\left( n,x,y,z;\overrightarrow{0}%
,r,-a,0\right) \frac{t^{n}}{n!}=\frac{\left( -1\right) ^{z}t^{z}}{2^{z-1}}%
\sum\limits_{n=0}^{\infty }\mathcal{E}_{n}^{\left( C,z\right) }\left(
x,y;a\right) \frac{t^{n}}{n!},  \label{1Y2Ec}
\end{equation}%
where $\mathcal{E}_{n}^{\left( C,z\right) }\left( x,y;a\right) $ denote the
two parametric kinds of Apostol-Euler polynomials of order $z$, which are
defined by%
\begin{equation}
\left( \frac{2}{a\exp \left( t\right) +1}\right) ^{z}\exp \left( xt\right)
\cos \left( yt\right) =\sum\limits_{n=0}^{\infty }\mathcal{E}_{n}^{\left(
C,z\right) }\left( x,y;a\right) \frac{t^{n}}{n!}  \label{Eulerrr}
\end{equation}%
(\textit{cf}. \cite{kizilates}).

Comparing the coefficients of $\frac{t^{n}}{n!}$ on both sides of equation (%
\ref{1Y2Ec}), we get the following result:

\begin{corollary}
Let $r$-tuples $\overrightarrow{0}=(0,0,...,0)$. Then we have%
\begin{equation}
\mathfrak{h}_{1}\left( n,x,y,z;\overrightarrow{0},r,-a,0\right) =\frac{%
\left( -1\right) ^{z}\left( n\right) ^{\underline{z}}}{2^{z-1}}\mathcal{E}%
_{n-z}^{\left( C,z\right) }\left( x,y;a\right) .  \label{1Y2Ecc}
\end{equation}
\end{corollary}

\begin{remark}
When $a=1$ and $z=1$, the polynomials $\mathfrak{h}_{1}\left( n,x,y,z;%
\overrightarrow{0},r,-a,0\right) $ reduce to the polynomials the$\mathcal{\ }
$cosine-Euler polynomials:%
\[
\mathfrak{h}_{1}\left( n,x,y,1;\overrightarrow{0},r,-1,0\right) =-n\mathcal{E%
}_{n-1}^{\left( C,1\right) }\left( x,y;1\right) =-nE_{n-1}^{\left( C\right)
}\left( x,y\right) 
\]%
(\textit{cf}. \cite{KimRyoo}, \cite{masjed2018}). Setting $y=0$ in (\ref%
{Eulerrr}), the polynomials $\mathcal{E}_{n}^{\left( C,z\right) }\left(
x,y;a\right) $ reduce to the Apostol-Euler polynomials of order $z$:%
\[
\mathcal{E}_{n}^{\left( z\right) }\left( x;a\right) =\mathcal{E}_{n}^{\left(
C,z\right) }\left( x,0;a\right) 
\]%
(\textit{cf}. \cite{SrivastavaChoi2}, \cite{Srivastava2018}).
\end{remark}

By using (\ref{1Y2Ec}), we have%
\[
\sum\limits_{n=0}^{\infty }\mathfrak{h}_{1}\left( n,x,y,z;\overrightarrow{0}%
,r,-a,0\right) \frac{t^{n}}{n!}=\frac{\left( -1\right) ^{z}t^{z}}{2^{z-1}}%
\sum\limits_{n=0}^{\infty }\mathcal{E}_{n}^{\left( z\right) }\left(
x;a\right) \frac{t^{n}}{n!}\sum\limits_{n=0}^{\infty }\left( -1\right) ^{n}%
\frac{\left( yt\right) ^{2n}}{\left( 2n\right) !}.
\]%
Therefore%
\[
\sum\limits_{n=0}^{\infty }\mathfrak{h}_{1}\left( n,x,y,z;\overrightarrow{0}%
,r,-a,0\right) \frac{t^{n}}{n!}=\sum\limits_{n=0}^{\infty }\left( \left(
-1\right) ^{z}2^{1-z}\sum\limits_{j=0}^{\left[ \frac{n}{2}\right] }\left(
-1\right) ^{j}\binom{n}{2j}\left( n-2j\right) ^{\underline{z}}y^{2j}\mathcal{%
E}_{n-2j-z}^{\left( z\right) }\left( x;a\right) \right) \frac{t^{n}}{n!}.
\]%
Comparing the coefficients of $\frac{t^{n}}{n!}$ on both sides of the above
equation we arrive at the following result:

\begin{corollary}
Let $r$-tuples $\overrightarrow{0}=(0,0,...,0)$. Then we have%
\[
\mathfrak{h}_{1}\left( n,x,y,z;\overrightarrow{0},r,-a,0\right) =\left(
-1\right) ^{z}2^{1-z}\sum\limits_{j=0}^{\left[ \frac{n}{2}\right] }\left(
-1\right) ^{j}\binom{n}{2j}\left( n-2j\right) ^{\underline{z}}y^{2j}\mathcal{%
E}_{n-2j-z}^{\left( z\right) }\left( x;a\right) .
\]
\end{corollary}

By using Euler's formula, we modify and unify equation (\ref{M3}) as follows:%
\begin{eqnarray}
B_{1}\left( t,x,y,z,\overrightarrow{u},r,a,b\right)  &=&2\left( b+f\left(
t,a\right) \right) ^{z}\exp (xt)M_{5}\left( t,y,\overrightarrow{u},r\right) 
\label{M3a} \\
&=&\sum\limits_{n=0}^{\infty }\mathfrak{h}_{2}\left( n,w,z;\overrightarrow{u}%
,r,a,b\right) \frac{t^{n}}{n!}.  \nonumber
\end{eqnarray}%
where%
\begin{equation}
M_{5}\left( t,y,\overrightarrow{u},r\right) =\exp \left(
\sum\limits_{j=1}^{r}u_{j}t^{j}\right) \sin \left( yt\right)
=\sum\limits_{n=0}^{\infty }S_{n}\left( \overrightarrow{u},y;r\right) \frac{%
t^{n}}{n!}.  \label{1Y2s}
\end{equation}%
Observe that when $r=1$, (\ref{1Y2s}) reduces to the (\ref{6b}). Setting $y=%
\frac{\pi }{2}$ in (\ref{1Y2s}), we have%
\[
H_{n}\left( \overrightarrow{u},r\right) =S_{n}\left( \overrightarrow{u},%
\frac{\pi }{2};r\right) .
\]

\begin{theorem}
Let $\overrightarrow{u}=\left( u_{1},u_{2},u_{3},\ldots ,u_{r}\right) $.
Then we have%
\begin{equation}
\mathfrak{h}_{2}\left( n,x,y,z;\overrightarrow{u},r,a,b\right)
=2\sum\limits_{j=0}^{n}\binom{n}{j}y_{6}\left( n-j;x,z;a,b\right)
S_{j}\left( \overrightarrow{u},y;r\right) .  \label{2.4}
\end{equation}
\end{theorem}

\begin{proof}
Combining (\ref{y6}), (\ref{1Y2s}) and (\ref{M3a}), we have%
\[
\sum\limits_{n=0}^{\infty }\mathfrak{h}_{2}\left( n,x,y,z;\overrightarrow{u}%
,r,a,b\right) \frac{t^{n}}{n!}=2\sum\limits_{n=0}^{\infty }y_{6}\left(
n;x,z;a,b\right) \frac{t^{n}}{n!}\sum\limits_{n=0}^{\infty }S_{n}\left( 
\overrightarrow{u},y;r\right) \frac{t^{n}}{n!}.
\]%
Therefore%
\[
\sum\limits_{n=0}^{\infty }\mathfrak{h}_{2}\left( n,x,y,z;\overrightarrow{u}%
,r,a,b\right) \frac{t^{n}}{n!}=2\sum\limits_{n=0}^{\infty
}\sum\limits_{j=0}^{n}\binom{n}{j}y_{6}\left( n-j;x,z;a,b\right) S_{j}\left( 
\overrightarrow{u},y;r\right) \frac{t^{n}}{n!}.
\]%
Comparing the coefficients of $\frac{t^{n}}{n!}$ on both sides of the above
equation we arrive at the desired result.
\end{proof}

Substituting $b=0$,%
\[
f\left( t,a\right) =\frac{t}{a\exp \left( t\right) -1}
\]%
and $\overrightarrow{u}=\overrightarrow{0}$ into (\ref{M3a}), we have the
following equation:%
\begin{equation}
B_{1}\left( t,x,y,z,\overrightarrow{0},r,a,0\right) =2F_{BS}\left(
t,x,y;a,z\right)   \label{MBS}
\end{equation}%
where the function $F_{BS}\left( t,x,y;a,z\right) $ is a generating function
for the two parametric kinds of the Apostol-Bernoulli polynomials of order $z
$,%
\begin{equation}
F_{BS}\left( t,x,y;a,z\right) =\left( \frac{t}{a\exp \left( t\right) -1}%
\right) ^{z}\exp \left( xt\right) \sin \left( yt\right)
=\sum\limits_{n=0}^{\infty }\mathcal{B}_{n}^{\left( S,z\right) }\left(
x,y;a\right) \frac{t^{n}}{n!}  \label{bernoulliS}
\end{equation}%
(\textit{cf}. \cite{kizilates}). Thus, using (\ref{MBS}), we have the
following result:

\begin{corollary}
Let $r$-tuples $\overrightarrow{0}=(0,0,...,0)$. Then we have%
\[
\mathfrak{h}_{2}\left( n,x,y,z;\overrightarrow{0},r,a,0\right) =2\mathcal{B}%
_{n}^{\left( S,z\right) }\left( x,y;a\right) .
\]
\end{corollary}

\begin{remark}
When $a=1$ and $z=1$, the polynomials $\mathfrak{h}_{2}\left( n,x,y,z;%
\overrightarrow{0},r,a,0\right) $ reduces to the polynomials $\mathcal{B}%
_{n}^{\left( S,1\right) }\left( x,y;1\right) =B_{n}^{\left( S\right) }\left(
x,y\right) $, which denote sine-Bernoulli polynomials:%
\[
\mathfrak{h}_{2}\left( n,x,y,1;\overrightarrow{0},r,1,0\right)
=2B_{n}^{\left( S\right) }\left( x,y\right) 
\]%
(\textit{cf}. \cite{KimRyoo}). Setting $y=\frac{\pi }{2}$ in (\ref%
{bernoulliS}), we have%
\[
\mathcal{B}_{n}^{\left( z\right) }\left( x;a\right) =\mathcal{B}_{n}^{\left(
S,z\right) }\left( x,\frac{\pi }{2};a\right) 
\]%
(\textit{cf}. \cite{SrivastavaChoi2}, \cite{Srivastava2018}).
\end{remark}

By using (\ref{MBS}), we have%
\[
\sum\limits_{n=0}^{\infty }\mathfrak{h}_{2}\left( n,x,y,z;\overrightarrow{0}%
,r,a,0\right) \frac{t^{n}}{n!}=2\sum\limits_{n=0}^{\infty }\mathcal{B}%
_{n}^{\left( z\right) }\left( x;a\right) \frac{t^{n}}{n!}\sum\limits_{n=0}^{%
\infty }\left( -1\right) ^{n}\frac{\left( yt\right) ^{2n+1}}{\left(
2n+1\right) !}.
\]%
Therefore%
\[
\sum\limits_{n=0}^{\infty }\mathfrak{h}_{2}\left( n,x,y,z;\overrightarrow{0}%
,r,a,0\right) \frac{t^{n}}{n!}=\sum\limits_{n=0}^{\infty }\left(
2\sum\limits_{j=0}^{\left[ \frac{n-1}{2}\right] }\left( -1\right) ^{j}\binom{%
n}{2j+1}y^{2j+1}\mathcal{B}_{n-1-2j}^{\left( z\right) }\left( x;a\right)
\right) \frac{t^{n}}{n!}
\]

\begin{corollary}
Let $r$-tuples $\overrightarrow{0}=(0,0,...,0)$. Then we have%
\[
\mathfrak{h}_{2}\left( n,x,y,z;\overrightarrow{0},r,a,0\right)
=2\sum\limits_{j=0}^{\left[ \frac{n-1}{2}\right] }\left( -1\right) ^{j}%
\binom{n}{2j+1}y^{2j+1}\mathcal{B}_{n-1-2j}^{\left( z\right) }\left(
x;a\right) .
\]
\end{corollary}

We modify (\ref{MBS}) as follows:%
\begin{equation}
\sum\limits_{n=0}^{\infty }\mathfrak{h}_{2}\left( n,x,y,z;\overrightarrow{0}%
,r,-a,0\right) \frac{t^{n}}{n!}=\frac{\left( -1\right) ^{z}t^{z}}{2^{z-1}}%
\sum\limits_{n=0}^{\infty }\mathcal{E}_{n}^{\left( S,z\right) }\left(
x,y;a\right) \frac{t^{n}}{n!}  \label{MES}
\end{equation}%
where $\mathcal{E}_{n}^{\left( S,z\right) }\left( x,y;a\right) $ denote the
two parametric kinds of Apostol-Euler polynomials of order $z$, which are
defined by the following generating function:%
\begin{equation}
\left( \frac{2}{a\exp \left( t\right) +1}\right) ^{z}\exp \left( xt\right)
\sin \left( yt\right) =\sum\limits_{n=0}^{\infty }\mathcal{E}_{n}^{\left(
S,z\right) }\left( x,y;a\right) \frac{t^{n}}{n!}  \label{EulerS}
\end{equation}%
(\textit{cf}. \cite{kizilates}).

By using (\ref{MES}), we get the following result:

\begin{corollary}
Let $r$-tuples $\overrightarrow{0}=(0,0,...,0)$. Then we have%
\begin{equation}
\mathfrak{h}_{2}\left( n,x,y,z;\overrightarrow{0},r,-a,0\right) =\frac{%
\left( -1\right) ^{z}\left( n\right) ^{\underline{z}}}{2^{z-1}}\mathcal{E}%
_{n-z}^{\left( S,z\right) }\left( x,y;a\right) .  \label{MES1}
\end{equation}
\end{corollary}

\begin{remark}
When $a=1$ and $z=1$, the polynomials $\mathfrak{h}_{2}\left( n,x,y,z;%
\overrightarrow{0},r,-a,0\right) $ reduce to the polynomials $\mathcal{E}%
_{n-1}^{\left( S,1\right) }\left( x,y;1\right) =E_{n-1}^{\left( S\right)
}\left( x,y\right) $, which denote sine-Euler polynomials:%
\[
\mathfrak{h}_{2}\left( n,x,y,1;\overrightarrow{0},r,-1,0\right)
=-nE_{n-1}^{\left( S\right) }\left( x,y\right) 
\]%
(\textit{cf}. \cite{KimRyoo}, \cite{masjed2018}). Setting $y=\frac{\pi }{2}$
in (\ref{EulerS}), we have%
\[
\mathcal{E}_{n}^{\left( z\right) }\left( x;a\right) =\mathcal{E}_{n}^{\left(
S,z\right) }\left( x,\frac{\pi }{2};a\right) 
\]%
(\textit{cf}. \cite{SrivastavaChoi2}, \cite{Srivastava2018}).
\end{remark}

Using (\ref{MES}), we obtain%
\[
\sum\limits_{n=0}^{\infty }\mathfrak{h}_{2}\left( n,x,y,z;\overrightarrow{0}%
,r,-a,0\right) \frac{t^{n}}{n!}=\left( -t\right)
^{z}2^{1-z}\sum\limits_{n=0}^{\infty }\mathcal{E}_{n}^{\left( z\right)
}\left( x;a\right) \frac{t^{n}}{n!}\sum\limits_{n=0}^{\infty }\left(
-1\right) ^{n}\frac{\left( yt\right) ^{2n+1}}{\left( 2n+1\right) !}.
\]%
Therefore%
\begin{eqnarray*}
&&\sum\limits_{n=0}^{\infty }\mathfrak{h}_{2}\left( n,x,y,z;\overrightarrow{0%
},r,-a,0\right) \frac{t^{n}}{n!} \\
&=&\sum\limits_{n=0}^{\infty }\left( -1\right) ^{z}2^{1-z}\sum\limits_{j=0}^{
\left[ \frac{n-1}{2}\right] }\left( -1\right) ^{j}\binom{n}{2j+1}\left(
n-1-2j\right) ^{\underline{z}}y^{2j+1}\mathcal{E}_{n-1-z-2j}^{\left(
z\right) }\left( x;a\right) \frac{t^{n}}{n!}.
\end{eqnarray*}%
By using (\ref{Esh1}), we obtain the following result:

\begin{corollary}
Let $r$-tuples $\overrightarrow{0}=(0,0,...,0)$. Then we have%
\[
\mathfrak{h}_{2}\left( n,x,y,z;\overrightarrow{0},r,-a,0\right) =\left(
-1\right) ^{z}2^{1-z}\sum\limits_{j=0}^{\left[ \frac{n-1}{2}\right] }\left(
-1\right) ^{j}\binom{n}{2j+1}\left( n-1-2j\right) ^{\underline{z}}y^{2j+1}%
\mathcal{E}_{n-1-z-2j}^{\left( z\right) }\left( x;a\right) .
\]
\end{corollary}

\section{Relations among the polynomials $\mathcal{K}\left( n;w,\protect%
\overrightarrow{u},r\right) $, trigonometric functions and hypergeometric
function}

In this section, we study the following two variable polynomials%
\begin{equation}
N_{n}\left( w\right) =\mathcal{K}\left( n;w,\overrightarrow{0},r\right) 
\label{Nn77}
\end{equation}%
where $r$-tuples $\overrightarrow{0}=(0,0,\ldots ,0)$, the polynomials $%
\mathcal{K}\left( n;w,\overrightarrow{0},r\right) $ are given in equation (%
\ref{g1}). We set $N_{n}\left( w\right) =N_{n}(\left( x,y\right) )$. We
investigate some properties of the\ polynomials $N_{n}\left( w\right) $. We
give relations among the polynomials $N_{n}\left( w\right) $, trigonometric
functions and hypergeometric functions. The polynomials $N_{n}\left(
w\right) $ are also related to other special polynomials, such as the
Milne-Thomson-type polynomials and the generalized Hermite-Kamp\`{e} de F%
\`{e}riet polynomials.

A series representation of the polynomials $N_{n}\left( w\right) $ is given
by%
\begin{equation}
G(t,w)=\sum\limits_{n=0}^{\infty }N_{n}\left( w\right) \frac{t^{n}}{n!}=\exp
\left( wt\right) .  \label{Nw}
\end{equation}%
Alternative forms of the above generating functions are given as follows:%
\begin{equation}
G(t,w)=\text{ }_{0}F_{0}\left[ 
\begin{array}{c}
- \\ 
-%
\end{array}%
;wt\right] ,  \label{nny}
\end{equation}%
\begin{equation}
G(t,w)=\text{ }_{0}F_{0}\left[ 
\begin{array}{c}
- \\ 
-%
\end{array}%
;xt\right] \left\{ _{0}F_{1}\left[ 
\begin{array}{c}
- \\ 
\frac{1}{2}%
\end{array}%
;\frac{-y^{2}t^{2}}{4}\right] +iyt\text{ }_{0}F_{1}\left[ 
\begin{array}{c}
- \\ 
\frac{3}{2}%
\end{array}%
;\frac{-y^{2}t^{2}}{4}\right] \right\} ,  \label{YN}
\end{equation}%
and%
\[
N_{m}\left( w\right) =\frac{\partial ^{m}}{\partial t^{m}}\left\{ \text{ }%
_{0}F_{0}\left[ 
\begin{array}{c}
- \\ 
-%
\end{array}%
;wt\right] \right\} \mid _{t=0}, 
\]%
where $_{p}F_{q}\left[ 
\begin{array}{c}
\alpha _{1},\ldots ,\alpha _{p} \\ 
\beta _{1},\ldots ,\beta _{q}%
\end{array}%
;z\right] $ denotes hypergeometric function, defined by%
\begin{equation}
_{p}F_{q}\left[ 
\begin{array}{c}
\alpha _{1},\ldots ,\alpha _{p} \\ 
\beta _{1},\ldots ,\beta _{q}%
\end{array}%
;z\right] =\sum\limits_{m=0}^{\infty }\left( \frac{\prod\limits_{j=1}^{p}%
\left( \alpha _{j}\right) _{m}}{\prod\limits_{j=1}^{q}\left( \beta
_{j}\right) _{m}}\right) \frac{z^{m}}{m!}.  \label{1Y3}
\end{equation}

A series in (\ref{1Y3}) converges for all $z$ if $p<q+1$, and for $%
\left\vert z\right\vert <1$ if $p=q+1$ and also all $\beta _{j}$, $\left(
j=1,2,...,q\right) $ are real or complex parameters with $\beta _{j}\notin 
\mathbb{N}$ (\textit{cf.} \cite{Bailey}, \cite{Choi}, \cite{JohnPearson}, 
\cite{Tremblay}, \cite{Srivastava2}).

By using the above generating functions, we obtain the following well-known
identity:%
\begin{equation}
N_{n}\left( w\right) =\sum\limits_{j=0}^{n}\binom{n}{j}x^{n-j}(iy)^{j}=%
\left( x+iy\right) ^{n}.  \label{Nxy}
\end{equation}%
Replacing $w$ by $\overline{w}$, we modify (\ref{Nxy}), we have%
\begin{equation}
N_{n}\left( \overline{w}\right) =\sum\limits_{j=0}^{n}\binom{n}{j}%
x^{n-j}(-iy)^{j}=\left( x-iy\right) ^{n}.  \label{Nxya}
\end{equation}%
Observe that%
\[
N_{n}\left( w\right) N_{n}\left( \overline{w}\right) =\left\vert
w\right\vert ^{2n}. 
\]

By using the Riemann integral, we derive some identities and formulas
including the polynomials $N_{n}\left( w\right) $, the Bernoulli numbers and
other special polynomials.

\begin{theorem}
Let $n\in \mathbb{N}$ and $w=x+iy$. Then we have%
\begin{equation}
N_{n-1}\left( w\right) =\frac{\overline{w}}{x^{2}+y^{2}}C_{n}\left(
x,y\right) +\frac{iw+2y}{x^{2}+y^{2}}S_{n}\left( x,y\right) .  \label{1Y4a}
\end{equation}
\end{theorem}

\begin{proof}
Integrating both sides of equation (\ref{Nw}) from $0$ to $t$ with respect
to the variable $v$, we get%
\begin{equation}
\sum\limits_{n=0}^{\infty }N_{n}\left( w\right) \int\limits_{0}^{t}\frac{%
v^{n}}{n!}dv=\int\limits_{0}^{t}e^{xv}\cos \left( yv\right)
dv+i\int\limits_{0}^{t}e^{xv}\sin \left( yv\right) dv.  \label{1Y4}
\end{equation}%
After some elementary calculations in the above equation, then combining
with (\ref{6a}) and (\ref{6b}), respectively, we obtain%
\[
\sum\limits_{n=1}^{\infty }N_{n-1}\left( w\right) \frac{t^{n}}{n!}=\frac{%
\overline{w}}{x^{2}+y^{2}}\sum\limits_{n=1}^{\infty }C_{n}\left( x,y\right) 
\frac{t^{n}}{n!}+\frac{iw+2y}{x^{2}+y^{2}}\sum\limits_{n=1}^{\infty
}S_{n}\left( x,y\right) \frac{t^{n}}{n!}. 
\]%
Comparing the coefficients of $\frac{t^{n}}{n!}$ on both sides of the above
equation, we arrive at the desired result.
\end{proof}

\begin{theorem}
Let $n\in \mathbb{N}$ and $w=x+iy$. Then we have%
\begin{equation}
B_{n-1}^{(-1)}=\frac{w^{1-n}}{\left( x^{2}+y^{2}\right) n}\left( \overline{w}%
C_{n}\left( x,y\right) +\left( iw+2y\right) S_{n}\left( x,y\right) \right) .
\label{1Y4b}
\end{equation}
\end{theorem}

\begin{proof}
Using (\ref{1Y4}) and (\ref{Nw}), we get%
\[
\frac{G(t,w)-1}{w}=\frac{\exp (wt)-1}{w}. 
\]%
Therefore%
\[
\frac{\exp (wt)-1}{w}=\frac{\overline{w}}{x^{2}+y^{2}}\sum\limits_{n=1}^{%
\infty }C_{n}\left( x,y\right) \frac{t^{n}}{n!}+\frac{iw+2y}{x^{2}+y^{2}}%
\sum\limits_{n=1}^{\infty }S_{n}\left( x,y\right) \frac{t^{n}}{n!}. 
\]%
Combining the above equation with (\ref{ApostolB}), we obtain%
\[
\sum\limits_{n=1}^{\infty }nw^{n-1}B_{n-1}^{(-1)}\frac{t^{n}}{n!}=\frac{%
\overline{w}}{x^{2}+y^{2}}\sum\limits_{n=1}^{\infty }C_{n}\left( x,y\right) 
\frac{t^{n}}{n!}+\frac{iw+2y}{x^{2}+y^{2}}\sum\limits_{n=1}^{\infty
}S_{n}\left( x,y\right) \frac{t^{n}}{n!}. 
\]%
Comparing the coefficients of $\frac{t^{n}}{n!}$ on both sides of the above
equation, we arrive at the desired result.
\end{proof}

Combining (\ref{1Y4a}) with (\ref{1Y4b}), we arrive at the following
corollary:

\begin{corollary}
Let $n\in \mathbb{N}$ and $w=x+iy$. Then we have%
\begin{equation}
N_{n-1}\left( w\right) =nw^{n-1}B_{n-1}^{(-1)}.  \label{1Y4c}
\end{equation}
\end{corollary}

By using (\ref{Nw}), we get the following well-known identity for the
numbers $B_{n}^{(-1)}$:

\begin{corollary}
Let $n\in \mathbb{N}_{0}$. Then we have%
\begin{equation}
B_{n}^{\left( -1\right) }=\frac{1}{n+1}.  \label{1Y4d}
\end{equation}
\end{corollary}

\begin{remark}
In work of Srivastava (\textit{cf}. \cite[Eq. (7.17)]{Srivastava}), we have
the following well-known formula including the Stirling numbers of the
second kind and the Bernoulli numbers of order $-k$:%
\[
B_{n}^{\left( -k\right) }=\frac{1}{\binom{n+k}{k}}S_{2}(n+k,k). 
\]%
Substituting $k=1$ into the above formula, since $S_{2}(n+1,1)=1$, we also
arrive at (\ref{1Y4d}).
\end{remark}

The polynomials $\mathcal{K}\left( n;w,\overrightarrow{u},r\right) $ are
linear combinations of the polynomials $N_{n}\left( w\right) $, presented by
the following theorem.

\begin{theorem}
Let $\overrightarrow{u}=\left( u_{1},u_{2},u_{3},\ldots ,u_{r}\right) $ and $%
w=x+iy$. Then we have%
\begin{equation}
\mathcal{K}\left( n;w,\overrightarrow{u},r\right) =\sum\limits_{j=0}^{n}%
\binom{n}{j}H_{j}\left( \overrightarrow{u},r\right) N_{n-j}\left( w\right) .
\label{C1}
\end{equation}
\end{theorem}

\begin{proof}
By using (\ref{H2}), (\ref{g1}) and (\ref{Nw}), we derive the following
functional equation:%
\begin{equation}
\mathcal{G}\left( t,w,\overrightarrow{u},r\right) =G(t,w)F_{R}\left( t,%
\overrightarrow{u},r\right) .  \label{g11}
\end{equation}%
From the above equation, we have%
\[
\sum\limits_{n=0}^{\infty }\mathcal{K}\left( n;w,\overrightarrow{u},r\right) 
\frac{t^{n}}{n!}=\sum\limits_{n=0}^{\infty }N_{n}\left( w\right) \frac{t^{n}%
}{n!}\sum\limits_{n=0}^{\infty }H_{n}\left( \overrightarrow{u},r\right) 
\frac{t^{n}}{n!}. 
\]%
Therefore%
\[
\sum\limits_{n=0}^{\infty }\mathcal{K}\left( n;w,\overrightarrow{u},r\right) 
\frac{t^{n}}{n!}=\sum\limits_{n=0}^{\infty }\sum\limits_{j=0}^{n}\binom{n}{j}%
H_{j}\left( \overrightarrow{u},r\right) N_{n-j}\left( w\right) \frac{t^{n}}{%
n!}. 
\]%
Comparing the coefficients of $\frac{t^{n}}{n!}$ on both sides of the above
equation, we arrive at the desired result.
\end{proof}

Similarly, the polynomials $h\left( n,w,z;\overrightarrow{0},r,a,b\right) $, 
$h_{1}\left( n,w,z;\overrightarrow{0},r,a,b\right) $, and $h_{2}\left( n,w,z;%
\overrightarrow{0},r,a,b\right) $ are also linear combinations of the
polynomials $N_{n}\left( w\right) $, presented as follows:

Substituting $\overrightarrow{u}=\overrightarrow{0}$ into (\ref{NK3}), (\ref%
{NK1}) and (\ref{NK2}), after that combining the last equation with equation
(\ref{Nn77}), we arrive at the following identities, respectively:

\begin{corollary}
\[
h\left( n,w,z;\overrightarrow{0},r,a,b\right) =\sum\limits_{j=0}^{n}\binom{n%
}{j}y_{6}^{(z)}\left( n-j;a,b\right) N_{j}\left( w\right) ,
\]%
\[
h_{1}\left( n,w,z;\overrightarrow{0},r,a,b\right) =\sum\limits_{j=0}^{n}%
\binom{n}{j}y_{6}^{(z)}\left( n-j;a,b\right) \left( N_{j}\left( w\right)
+N_{j}\left( \overline{w}\right) \right) ,
\]%
and%
\[
h_{2}\left( n,w,z;\overrightarrow{0},r,a,b\right) =\sum\limits_{j=0}^{n}%
\binom{n}{j}y_{6}^{(z)}\left( n-j;a,b\right) \left( N_{j}\left( w\right)
-N_{j}\left( \overline{w}\right) \right) .
\]
\end{corollary}

\section{Relations among Hermite-type polynomials and Chebyshev-type
polynomials and Dickson polynomials}

In this section, we give relations among the Hermite-type polynomials, the
generalized Hermite-Kamp\`{e} de F\`{e}riet polynomials, and the Chebyshev
polynomials.

Let $w=x+iy$. By using (\ref{Nxy}), we modify (\ref{C1}) as follows:%
\[
\mathcal{K}\left( n;x+iy,\overrightarrow{u},r\right) =\sum\limits_{j=0}^{n}%
\binom{n}{j}H_{j}\left( \overrightarrow{u},r\right) \sum\limits_{k=0}^{n-j}%
\binom{n-j}{k}x^{n-j-k}\left( iy\right) ^{k}.
\]%
Therefore, we define the following polynomials:%
\[
P_{1}\left( n,x,y,\overrightarrow{u},r\right) =\func{Re}\left\{ \mathcal{K}%
\left( n;x+iy,\overrightarrow{u},r\right) \right\} 
\]%
and%
\[
P_{2}\left( n,x,y,\overrightarrow{u},r\right) =\func{Im}\left\{ \mathcal{K}%
\left( n;x+iy,\overrightarrow{u},r\right) \right\} .
\]%
Explicit formulas for these polynomials are given as follows:%
\begin{equation}
P_{1}\left( n,x,y,\overrightarrow{u},r\right) =\sum\limits_{j=0}^{n}\binom{n%
}{j}H_{j}\left( \overrightarrow{u},r\right) \sum\limits_{k=0}^{\left[ \frac{%
n-j}{2}\right] }\left( -1\right) ^{k}\binom{n-j}{2k}x^{n-j-2k}y^{2k}
\label{R1}
\end{equation}%
and%
\begin{equation}
P_{2}\left( n,x,y,\overrightarrow{u},r\right) =\sum\limits_{j=0}^{n}\binom{n%
}{j}H_{j}\left( \overrightarrow{u},r\right) \sum\limits_{k=0}^{\left[ \frac{%
n-j-1}{2}\right] }\left( -1\right) ^{k}\binom{n-j}{2k+1}x^{n-j-2k-1}y^{2k+1}.
\label{I1}
\end{equation}

Combining equations (\ref{R1}) and (\ref{I1}) with (\ref{7a}) and (\ref{7b}%
), respectively, we arrive at the following theorem:

\begin{theorem}
Let $\overrightarrow{u}=\left( u_{1},u_{2},u_{3},\ldots ,u_{r}\right) $.
Then we have%
\begin{equation}
P_{1}\left( n,x,y,\overrightarrow{u},r\right) =\sum\limits_{j=0}^{n}\binom{n%
}{j}H_{j}\left( \overrightarrow{u},r\right) C_{n-j}\left( x,y\right)
\label{R1A}
\end{equation}%
and%
\begin{equation}
P_{2}\left( n,x,y,\overrightarrow{u},r\right) =\sum\limits_{j=0}^{n}\binom{n%
}{j}H_{j}\left( \overrightarrow{u},r\right) S_{n-j}\left( x,y\right) .
\label{l1A}
\end{equation}
\end{theorem}

Substituting $y=\sqrt{1-x^{2}}$ into (\ref{R1}) and (\ref{I1}), we obtain
relations among the Chebyshev polynomials of the first kind $T_{n}\left(
x\right) $, the Chebyshev polynomials of the second kind $U_{n}\left(
x\right) $, the generalized Hermite-Kamp\`{e} de F\`{e}riet polynomials $%
H_{n}\left( \overrightarrow{u},r\right) $ by the following theorem:

\begin{theorem}
Let $\overrightarrow{u}=\left( u_{1},u_{2},u_{3},\ldots ,u_{r}\right) $.
Then we have%
\[
P_{3}(n,x,\overrightarrow{u},r)=\sum\limits_{j=0}^{n}\binom{n}{j}H_{j}\left( 
\overrightarrow{u},r\right) T_{n-j}\left( x\right) , 
\]%
where%
\[
P_{3}(n,x,\overrightarrow{u},r)=P_{1}\left( n,x,\sqrt{1-x^{2}},%
\overrightarrow{u},r\right) 
\]

and%
\[
P_{4}(n,x,\overrightarrow{u},r)=\sum\limits_{j=0}^{n}\binom{n}{j}H_{j}\left( 
\overrightarrow{u},r\right) U_{n-j-1}\left( x\right) 
\]%
where%
\[
P_{4}(n,x,\overrightarrow{u},r)=\frac{P_{2}\left( n,x,\sqrt{1-x^{2}},%
\overrightarrow{u},r\right) }{\sqrt{1-x^{2}}}. 
\]
\end{theorem}

Using the polynomials $P_{3}(n,x,\overrightarrow{u},r)$ and $P_{4}(n,x,%
\overrightarrow{u},r)$, we arrive at the following corollary:

\begin{corollary}
Let $n\in \mathbb{N}_{0}$. Then we have%
\begin{equation}
T_{n}\left( x\right) =C_{n}\left( x,\sqrt{1-x^{2}}\right) =\func{Re}\left\{
N_{n}\left( \left( x,\sqrt{1-x^{2}}\right) \right) \right\}   \label{CT}
\end{equation}%
and for $n\geq 1$%
\begin{equation}
U_{n-1}\left( x\right) =\frac{S_{n}\left( x,\sqrt{1-x^{2}}\right) }{\sqrt{%
1-x^{2}}}=\frac{1}{\sqrt{1-x^{2}}}\func{Im}\left\{ N_{n}\left( \left( x,%
\sqrt{1-x^{2}}\right) \right) \right\} .  \label{SU}
\end{equation}
\end{corollary}

By (\ref{Dickson}), (\ref{dickson2}), (\ref{CT}) and (\ref{SU}), we obtain
the following result which are related to the Dickson polynomials, the
polynomials $C_{n}\left( x,y\right) $, $S_{n}\left( x,y\right) $ and the
polynomials $N_{n}\left( w\right) $:

\begin{corollary}
Let $n\in \mathbb{N}_{0}$. Then we have%
\[
D_{n}\left( 2x,1\right) =2C_{n}\left( x,\sqrt{1-x^{2}}\right) =2\func{Re}%
\left\{ N_{n}\left( \left( x,\sqrt{1-x^{2}}\right) \right) \right\} .
\]
\end{corollary}

\begin{corollary}
Let $n\in \mathbb{N}$. Then we have%
\[
\mathfrak{E}_{n-1}\left( 2x,1\right) =\frac{S_{n}\left( x,\sqrt{1-x^{2}}%
\right) }{\sqrt{1-x^{2}}}=\frac{1}{\sqrt{1-x^{2}}}\func{Im}\left\{
N_{n}\left( \left( x,\sqrt{1-x^{2}}\right) \right) \right\} .
\]
\end{corollary}

\section{Identities and relations including Chebyshev polynomials and
trigonometric polynomials}

Here, we give some identities and formulas which are relations among the
Chebyshev polynomials, the Dickson polynomials, the Bernoulli numbers, the
Euler numbers the Stirling numbers and other special polynomials.

Substituting $y=\sqrt{1-x^{2}}$ into Theorem 2.9 of \cite{kilar}, then
combining (\ref{CT}) and (\ref{SU}), we have the following result:

\begin{corollary}
Let $n\in 
\mathbb{N}
$. Then we have%
\begin{equation}
U_{n-1}\left( x\right) =2^{1-n}\sum_{j=1}^{n}\binom{n}{j}U_{j-1}\left(
x\right) T_{n-j}\left( x\right) .  \label{UT}
\end{equation}
\end{corollary}

By using (\ref{Dickson}), (\ref{dickson2}) and (\ref{UT}), we also obtain
the following result:

\begin{corollary}
Let $n\in 
\mathbb{N}
$. Then we have%
\[
\mathfrak{E}_{n-1}\left( 2x,1\right) =2^{-n}\sum_{j=1}^{n}\binom{n}{j}%
\mathfrak{E}_{j-1}\left( 2x,1\right) D_{n-j}\left( 2x,1\right) . 
\]
\end{corollary}

Substituting $y=\sqrt{1-x^{2}}$ into Theorem 1 of \cite{KimRyoo}, we have%
\[
E_{n}^{\left( C\right) }\left( x,\sqrt{1-x^{2}}\right) =\sum\limits_{j=0}^{n}%
\binom{n}{j}C_{j}\left( x,\sqrt{1-x^{2}}\right) E_{n-j} 
\]%
and%
\[
E_{n}^{\left( S\right) }\left( x,\sqrt{1-x^{2}}\right) =\sum\limits_{j=1}^{n}%
\binom{n}{j}S_{j}\left( x,\sqrt{1-x^{2}}\right) E_{n-j}. 
\]

Combining above equations with (\ref{CT}) and (\ref{SU}), respectively, we
obtain the following results:

\begin{corollary}
Let $n\in \mathbb{N}_{0}$. Then we have%
\begin{equation}
E_{n}^{\left( C\right) }\left( x,\sqrt{1-x^{2}}\right) =\sum\limits_{j=0}^{n}%
\binom{n}{j}T_{j}\left( x\right) E_{n-j}.  \label{ECT}
\end{equation}
\end{corollary}

\begin{corollary}
Let $n\in \mathbb{N}$. Then we have%
\begin{equation}
E_{n}^{\left( S\right) }\left( x,\sqrt{1-x^{2}}\right) =\sqrt{1-x^{2}}%
\sum\limits_{j=1}^{n}\binom{n}{j}U_{j-1}\left( x\right) E_{n-j}.  \label{ESU}
\end{equation}
\end{corollary}

By using (\ref{Dickson}), (\ref{dickson2}), (\ref{ECT}) and (\ref{ESU}), we
also obtain the following result:

\begin{corollary}
\[
E_{n}^{\left( C\right) }\left( x,\sqrt{1-x^{2}}\right) =\frac{1}{2}%
\sum\limits_{j=0}^{n}\binom{n}{j}D_{j}\left( 2x,1\right) E_{n-j} 
\]%
and%
\[
E_{n}^{\left( S\right) }\left( x,\sqrt{1-x^{2}}\right) =\sqrt{1-x^{2}}%
\sum\limits_{j=1}^{n}\binom{n}{j}\mathfrak{E}_{j-1}\left( 2x,1\right)
E_{n-j}. 
\]
\end{corollary}

On the other hand, substituting $y=\sqrt{1-x^{2}}$ into Theorem 6 of \cite%
{KimRyoo}, we have%
\[
B_{n}^{\left( C\right) }\left( x,\sqrt{1-x^{2}}\right) =\sum\limits_{j=0}^{n}%
\binom{n}{j}C_{j}\left( x,\sqrt{1-x^{2}}\right) B_{n-j} 
\]%
and%
\[
B_{n}^{\left( S\right) }\left( x,\sqrt{1-x^{2}}\right) =\sum\limits_{j=1}^{n}%
\binom{n}{j}S_{j}\left( x,\sqrt{1-x^{2}}\right) B_{n-j}. 
\]%
Combining above equations with (\ref{CT}) and (\ref{SU}), respectively, we
obtain the following results:

\begin{corollary}
Let $n\in \mathbb{N}_{0}$. Then we have%
\begin{equation}
B_{n}^{\left( C\right) }\left( x,\sqrt{1-x^{2}}\right) =\sum\limits_{j=0}^{n}%
\binom{n}{j}T_{j}\left( x\right) B_{n-j}.  \label{BCT}
\end{equation}
\end{corollary}

\begin{corollary}
Let $n\in \mathbb{N}$. The we have%
\begin{equation}
B_{n}^{\left( S\right) }\left( x,\sqrt{1-x^{2}}\right) =\sqrt{1-x^{2}}%
\sum\limits_{j=1}^{n}\binom{n}{j}U_{j-1}\left( x\right) B_{n-j}.  \label{BSU}
\end{equation}
\end{corollary}

By using (\ref{Dickson}), (\ref{dickson2}), (\ref{BCT}) and (\ref{BSU}) we
also obtain the following result:

\begin{corollary}
\[
B_{n}^{\left( C\right) }\left( x,\sqrt{1-x^{2}}\right) =\frac{1}{2}%
\sum\limits_{j=0}^{n}\binom{n}{j}D_{j}\left( 2x,1\right) B_{n-j} 
\]%
and%
\[
B_{n}^{\left( S\right) }\left( x,\sqrt{1-x^{2}}\right) =\sqrt{1-x^{2}}%
\sum\limits_{j=1}^{n}\binom{n}{j}\mathfrak{E}_{j-1}\left( 2x,1\right)
B_{n-j}. 
\]
\end{corollary}

By applying derivative operator to (\ref{6a}) and (\ref{6b}) with respect to 
$x$, then with respect to $y$, we obtain the following partial differential
equation:%
\[
\frac{\partial }{\partial x}F_{C}\left( t,x,y\right)
=\sum\limits_{n=0}^{\infty }\frac{\partial }{\partial x}C_{n}\left(
x,y\right) \frac{t^{n}}{n!}=\sum\limits_{n=0}^{\infty }nC_{n-1}\left(
x,y\right) \frac{t^{n}}{n!}. 
\]%
Comparing the coefficients of $\frac{t^{n}}{n!}$ on both sides of the above
equation, we have%
\[
\frac{\partial }{\partial x}C_{n}\left( x,y\right) =nC_{n-1}\left(
x,y\right) . 
\]%
Similarly, for $n\geq 1$, we have%
\[
\frac{\partial }{\partial x}S_{n}\left( x,y\right) =nS_{n-1}\left(
x,y\right) , 
\]%
\begin{equation}
\frac{\partial }{\partial y}C_{n}\left( x,y\right) =-nS_{n-1}\left(
x,y\right) ,  \label{A1}
\end{equation}%
\[
\frac{\partial }{\partial y}S_{n}\left( x,y\right) =nC_{n-1}\left(
x,y\right) . 
\]%
Substituting $y=\sqrt{1-x^{2}}$ into (\ref{A1}), we obtain the following
well-known identity as follows:%
\[
T_{n}^{\prime }\left( x\right) =nU_{n-1}\left( x\right) 
\]%
(\textit{cf.} \cite{Rivlin}).

\begin{theorem}
Let $n\geq 2$. Then we have%
\[
\frac{\partial ^{2}}{\partial x\partial y}C_{n}\left( x,y\right) =-n\left(
n-1\right) S_{n-2}\left( x,y\right) 
\]%
and%
\[
\frac{\partial ^{2}}{\partial x\partial y}S_{n}\left( x,y\right) =n\left(
n-1\right) C_{n-2}\left( x,y\right) . 
\]
\end{theorem}

\begin{proof}
By applying derivative operator to (\ref{6a}) and (\ref{6b}) with respect to 
$x$ and $y$, we obtain the following partial differential equations,
respectively:%
\[
\frac{\partial ^{2}}{\partial x\partial y}F_{C}\left( t,x,y\right)
=-t^{2}F_{S}\left( t,x,y\right) 
\]%
and%
\[
\frac{\partial ^{2}}{\partial x\partial y}F_{S}\left( t,x,y\right)
=t^{2}F_{C}\left( t,x,y\right) . 
\]%
From the above functional equations, we obtain%
\[
\sum\limits_{n=0}^{\infty }\frac{\partial ^{2}}{\partial x\partial y}%
C_{n}\left( x,y\right) \frac{t^{n}}{n!}=-\sum\limits_{n=0}^{\infty }n\left(
n-1\right) S_{n-2}\left( x,y\right) \frac{t^{n}}{n!} 
\]%
and%
\[
\sum\limits_{n=0}^{\infty }\frac{\partial ^{2}}{\partial x\partial y}%
S_{n}\left( x,y\right) \frac{t^{n}}{n!}=\sum\limits_{n=0}^{\infty }n\left(
n-1\right) C_{n-2}\left( x,y\right) \frac{t^{n}}{n!}. 
\]%
\bigskip Comparing the coefficients of $\frac{t^{n}}{n!}$ on both sides of
the above equations, we arrive at the desired result.
\end{proof}

\begin{theorem}
Let $n\in \mathbb{N}_{0}$. Then we have%
\begin{equation}
C_{n+1}\left( x,y\right) =xC_{n}\left( x,y\right) -yS_{n}\left( x,y\right) .
\label{s1}
\end{equation}
\end{theorem}

\begin{proof}
By applying derivative operator to (\ref{6a}) with respect to $t$, we obtain
the following partial differential equation:%
\[
\frac{\partial }{\partial t}F_{C}\left( t,x,y\right) =xF_{C}\left(
t,x,y\right) -yF_{S}\left( t,x,y\right) . 
\]%
From the above equation, we have%
\[
\sum\limits_{n=0}^{\infty }C_{n+1}\left( x,y\right) \frac{t^{n}}{n!}%
=x\sum\limits_{n=0}^{\infty }C_{n}\left( x,y\right) \frac{t^{n}}{n!}%
-y\sum\limits_{n=0}^{\infty }S_{n}\left( x,y\right) \frac{t^{n}}{n!}. 
\]%
Comparing the coefficients of $\frac{t^{n}}{n!}$ on both sides of the above
equation, we arrive at the desired result.
\end{proof}

\begin{remark}
By using (\ref{CT}), (\ref{SU}) and (\ref{s1}), we arrive at the equation (%
\ref{TileU2}).
\end{remark}

\begin{theorem}
Let $n\in \mathbb{N}_{0}$. Then we have%
\begin{equation}
S_{n+1}\left( x,y\right) =xS_{n}\left( x,y\right) +yC_{n}\left( x,y\right) .
\label{s2}
\end{equation}
\end{theorem}

\begin{proof}
By applying derivative operator to (\ref{6b}) with respect to $t$, we obtain
the following partial differential equation:%
\[
\frac{\partial }{\partial t}F_{S}\left( t,x,y\right) =xF_{S}\left(
t,x,y\right) +yF_{C}\left( t,x,y\right) . 
\]%
From the above equation, we get%
\[
\sum\limits_{n=0}^{\infty }S_{n+1}\left( x,y\right) \frac{t^{n}}{n!}%
=x\sum\limits_{n=0}^{\infty }S_{n}\left( x,y\right) \frac{t^{n}}{n!}%
+y\sum\limits_{n=0}^{\infty }C_{n}\left( x,y\right) \frac{t^{n}}{n!}. 
\]%
Comparing the coefficients of $\frac{t^{n}}{n!}$ on both sides of the above
equation, we arrive at the desired result.
\end{proof}

\begin{remark}
By using (\ref{CT}), (\ref{SU}) and (\ref{s2}), we arrive at the equation (%
\ref{TileU}). On the other hand, multiplying (\ref{s1}) by $x$ and (\ref{s2}%
) by $y$ and then side-by-side adding, and multiplying (\ref{s1}) by $y$ and
(\ref{s2}) by $x$ and then side-by-side subtracting, then after some
calculation, we arrive at the equation (\ref{1Y4a}).
\end{remark}

\end{document}